\documentclass[12pt, reqno, a4paper]{article}
\usepackage{amsmath,amssymb,amsthm, todonotes,url,mathtools}
\usepackage{thmtools}
\usepackage{thm-restate}
\usepackage[lmargin=35mm,rmargin=35mm,bmargin=30mm,tmargin=30mm]{geometry}
\usepackage{bbm}
\usepackage[inline]{enumitem}
\usepackage{tikz}
\usetikzlibrary{decorations.pathreplacing}
\usepackage{subcaption}
\usepackage{authblk}

\theoremstyle{plain} 
\newtheorem{theorem}{Theorem}

\newtheorem{lemma}[theorem]{Lemma}
\newtheorem{corollary}[theorem]{Corollary}

\newtheorem*{claim*}{Claim}
\newtheorem{problem}{Problem}
\newtheorem{conjecture}[theorem]{Conjecture}
\newtheorem*{conjecture*}{Conjecture}

\theoremstyle{definition}

\theoremstyle{remark}   

\newtheorem*{remark*}{Remark}

\theoremstyle{remark}

\newcommand{\Hext}{H_{\textrm{ext}}}
\newcommand{\calHext}{\mathcal{H}_{\textrm{ext}}}
\newcommand{\Cext}{C_{\textrm{ext}}}
\newcommand{\Dext}{D_{\textrm{ext}}}

\newcommand{\floor}[1]{\left\lfloor #1 \right\rfloor}
\newcommand{\ceil}[1]{\left\lceil #1 \right\rceil}

\def\rank{\operatorname{rank}}

\newcommand{\N}{\mathbb{N}}

\usepackage{dsfont}
\newcommand{\eps}{\varepsilon}

\newcommand{\midtilde}[1]{\mkern 1.5mu\widetilde{\mkern-1.5mu#1\mkern-1.5mu}\mkern 1.5mu}

\tikzstyle{vertex}=[circle, draw, fill=black, inner sep=0pt, minimum width=4pt]

\title{Reconstruction from smaller cards}
\author{Carla Groenland\footnote{Delft Institute of Applied Mathematics, Technische Universiteit Delft, 2628 CD Delft, Netherlands, \texttt{c.e.groenland@tudelft.nl}.}\quad
Tom Johnston\footnote{Mathematical Institute, University of Oxford, Oxford OX2 6GG, UK\\ \texttt{tom.johnston@bristol.ac.uk,  scott@maths.ox.ac.uk, jane.tan@maths.ox.ac.uk}}\quad
Alex Scott\protect\footnotemark[2] \footnote{Supported by EPSRC grant EP/V007327/1.}\quad
Jane Tan\protect\footnotemark[2]}
\date{}



\begin{document}

\maketitle

\begin{abstract}
The $\ell$-deck of a graph $G$ is the multiset of all induced subgraphs of $G$ on $\ell$ vertices. 
We say that a graph is reconstructible from its $\ell$-deck if no other graph has the same $\ell$-deck.  In 1957, Kelly showed that every tree with $n\ge3$ vertices can be reconstructed from its $(n-1)$-deck, and Giles strengthened this in 1976, proving that trees on at least 6 vertices can be reconstructed from their $(n-2)$-decks.  Our main theorem states that trees are reconstructible from their $(n-r)$-decks for all $r\le n/{9}+o(n)$, making substantial progress towards a conjecture of N\'ydl from 1990. In addition, we can recognise the connectedness of a graph from its $\ell$-deck when $\ell\ge 9n/10$, and reconstruct the degree sequence when $\ell\ge\sqrt{2n\log(2n)}$.
All of these results are significant improvements on previous bounds.
\end{abstract}

\section{Introduction}

Throughout this paper, all graphs are finite and undirected with no loops or multiple edges. 
Given a graph $G$ and any vertex $v\in V(G)$, the \textit{card} $G-v$ is the subgraph of $G$ obtained by removing the vertex $v$ together with all edges incident to $v$. The \emph{deck} $\mathcal{D}(G)$ is then the multiset of all unlabelled cards of $G$. 
A graph $G$ is said to be \emph{reconstructible} from its deck if any graph with the same deck is isomorphic to $G$. 

The graph reconstruction conjecture of Kelly and Ulam~\cite{K42,K57,U60} states that all graphs on at least three vertices are reconstructible. While this classical conjecture has been verified for certain classes such as trees~(Kelly \cite{K57}), outerplanar graphs~(Giles \cite{Giles74}) and maximal planar graphs~(Lauri \cite{Lauri81}), it remains open even for simple classes of graphs such as planar graphs and graphs of bounded maximum degree. However, various graph parameters, such as the degree sequence and connectedness, are known to be \textit{reconstructible} for general graphs in the sense that they are determined by the deck (i.e.~if two graphs have the same deck, then the parameter takes the same value for both).

There is a significant body of research on the problem of reconstructing graphs and graph parameters from smaller cards: instead of taking induced subgraphs on $n-1$ vertices, it is natural to consider cards which are the induced subgraphs on $\ell$ vertices where $\ell$ may be much smaller than $n-1$. The \emph{$\ell$-deck} of $G$, denoted by $\mathcal{D}_\ell(G)$, is the multiset of all induced subgraphs of $G$ on $\ell$ vertices (in this notation $\mathcal{D}(G)=\mathcal{D}_{n-1}(G)$).  
A graph or graph parameter is \emph{reconstructible from the $\ell$-deck} if it is determined by the $\ell$-deck. 

Intuitively, individual cards that are smaller carry less information. Indeed, the $(\ell-1)$-deck is determined by the $\ell$-deck for each $\ell$, as can be shown by a simple counting argument (see Lemma \ref{lem:Kelly}). Thus, if a graph is reconstructible from its $\ell'$-deck then it is reconstructible from its $\ell$-deck for all $\ell\geq \ell'$. The main question is then to determine the threshold; that is, to determine the smallest $\ell$ for which a given class of graphs or a property is reconstructible from the $\ell$-deck. 

Reconstruction from small cards is generally attributed to Kelly, although the strengthening of the Reconstruction Conjecture that follows seems to be formulated by Manvel (who calls it ``Kelly's Conjecture'').
\begin{conjecture}[\cite{K57, Manvel74}]
For every $r \in \mathbb{N}$, there is an integer $N_r$ such that every graph with at least $N_r$ vertices is reconstructible from its $(n-r)$-deck.
\end{conjecture}
Kelly and Ulam's conjecture posits that $N_1=3$. This stronger conjecture did not receive much attention until 1974 when it was studied by Manvel \cite{Manvel74}, who showed that several classes of graphs, such as connected graphs, trees, regular graphs and bipartite graphs, can be \emph{recognised} from the $(n-2)$-deck where $n\geq 6$ is the number of vertices (that is, if $G$ and $H$ are graphs with $n$ vertices and the same $(n-2)$-deck, then either both graphs or neither belong to the class).  Since then, recognition and reconstruction problems of this type have been widely studied. Recent developments include the reconstructibility of 3-regular $n$-vertex graphs from the $(n-2)$-deck~(Kostochka, Nahvi, West and Zirlin \cite{kostochka20193}) and that almost all graphs are reconstructible from $\binom{r + 2}{2}$ specially chosen cards from the $(n - r)$-deck when $r \leq (1/2 - o(1))n$ (Spinoza and West \cite{SW19}, building on results of M\"uller \cite{M76} and Bollob\'as \cite{Boll90}). For further background, we refer to the survey of Kostochka and West~\cite{KW21}.

For general graphs, it is not possible to guarantee reconstructibility from the $(n-r)$-deck unless $r=o(n)$, as shown by the following theorem of N\'{y}dl.
\begin{theorem}[N\'{y}dl \cite{Nydl92}]\label{n1}
For any integer $n_0$ and $0<\alpha<1$, there exists an integer $n > n_0$ such that there are two non-isomorphic graphs on $n$ vertices which share the same multiset of subgraphs of order at most $\alpha n$.
\end{theorem}

However, 
N\'ydl's theorem may not hold for specific families of graphs 
such as the class of trees. In fact, N\'{y}dl conjectured in 1990 that no two non-isomorphic trees have the same $\ell$-deck when $\ell$ is slightly larger than $n/2$.

\begin{conjecture}[N\'ydl \cite{nydl1990note}]
\label{conj:trees}
For any $n\geq 4$ and $\ell \geq \lfloor n/2 \rfloor+1$, any two trees on $n$ vertices with the same $\ell$-deck are isomorphic.
\end{conjecture}

The conjectured bound would be sharp: N\'ydl \cite{nydl1990note} presented trees for which $\ell\ge \lfloor n/2 \rfloor+1$ is necessary (see~\cite{KW21} for a short proof).

There has been no progress on N\'ydl's conjecture since it was made in \cite{nydl1990note}.  Indeed, the best previous result is an earlier bound of Giles \cite{giles1976reconstructing} from 1976, which states that for $n \geq 5$ no two non-isomorphic $n$-vertex trees have the same $(n-2)$-deck. Using the result of Manvel \cite{Manvel74} that the class of $n$-vertex trees is recognisable from the $(n - 2)$-deck when $n \geq 6$, Giles' result confirms that trees can be reconstructed from their $(n-2)$-deck for all $n \geq 6$.

Our main theorem improves very substantially on the result of Giles and takes a significant step towards Conjecture \ref{conj:trees}, showing that we can reconstruct trees from the $(n-r)$-deck for $r$ with linear size. 

\begin{restatable}{theorem}{trees}
\label{thm:trees}
Any $n$-vertex tree $T$ can be reconstructed from $\mathcal{D}_{n-r}(T)$ when $r < \frac{n}{9}-\frac{4}{9}\sqrt{8n+5}-1$.
\end{restatable}

In particular, it follows that N\'ydl's theorem (Theorem \ref{n1}) does not hold when restricted to the class of trees.
We remark that Conjecture \ref{conj:trees} is false in the case $n=13$, as demonstrated by the two graphs in Figure \ref{fig:counter_example} which have been verified to have the same deck by computer. However, our computer search has also shown that the conjecture is true for all other values $4 \leq n\leq 25$, and it remains open for large $n$.

\begin{figure}
\centering
\begin{subfigure}{0.5 \textwidth}
    \centering
    \scalebox{0.85}{
    \begin{tikzpicture}[xscale=0.7]
        \draw node[style=vertex](0) at (0,0) {};
        \draw node[style=vertex](1) at (1,0) {};
        \draw node[style=vertex](2) at (2,0) {};
        \draw node[style=vertex](3) at (3,0) {};
        \draw node[style=vertex](4) at (4,0) {};
        \draw node[style=vertex](5) at (5,0) {};
        \draw node[style=vertex](6) at (6,0) {};
        \draw node[style=vertex](7) at (7,0) {};
        \draw node[style=vertex](8) at (8,0) {};
        
        \draw node(9)[style=vertex] at (2,-0.7) {};
        \draw node(10)[style=vertex] at (3,-0.7) {};
        \draw node(11)[style=vertex] at (3,-1.4) {};
        \draw node(12)[style=vertex] at (5,-0.7) {};
        
        \draw (0) -- (1);
        \draw (1) -- (2);
        \draw (2) -- (3);
        \draw (3) -- (4);
        \draw (4) -- (5);
        \draw (5) -- (6);
        \draw (6) -- (7);
        \draw (7) -- (8);
        \draw (2) -- (9);
        \draw (3) -- (10);
        \draw (10) -- (11);
        \draw (5) -- (12);
    \end{tikzpicture}}
\end{subfigure}%
\begin{subfigure}{0.5 \textwidth}
    \centering
    \scalebox{0.85}{\begin{tikzpicture}[xscale=0.7]
        \draw node[style=vertex](0) at (0,0) {};
        \draw node[style=vertex](1) at (1,0) {};
        \draw node[style=vertex](2) at (2,0) {};
        \draw node[style=vertex](3) at (3,0) {};
        \draw node[style=vertex](4) at (4,0) {};
        \draw node[style=vertex](5) at (5,0) {};
        \draw node[style=vertex](6) at (6,0) {};
        \draw node[style=vertex](7) at (7,0) {};
        \draw node[style=vertex](8) at (8,0) {};
        
        \draw node(9)[style=vertex] at (2,-0.7) {};
        \draw node(10)[style=vertex] at (4,-0.7) {};
        \draw node(11)[style=vertex] at (5,-0.7) {};
        \draw node(12)[style=vertex] at (5,-1.4) {};
        
        \draw (0) -- (1);
        \draw (1) -- (2);
        \draw (2) -- (3);
        \draw (3) -- (4);
        \draw (4) -- (5);
        \draw (5) -- (6);
        \draw (6) -- (7);
        \draw (7) -- (8);
        \draw (2) -- (9);
        \draw (4) -- (10);
        \draw (5) -- (11);
        \draw (11) -- (12);
    \end{tikzpicture}}
\end{subfigure}

    \caption{Two non-isomorphic trees on 13 vertices which have the same 7-deck.}
    \label{fig:counter_example}
\end{figure}
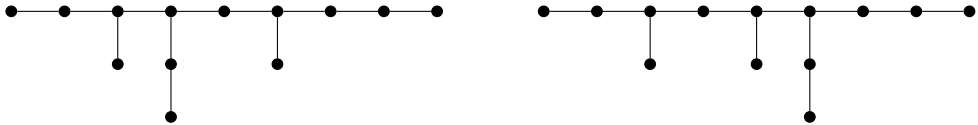

It is worth noting that the class of trees, being one of the first non-trivial classes shown to be reconstructible in the classical sense, is very prominent in reconstruction literature. For example, assuming we know \emph{a priori} that the graph is a tree, Harary and Palmer \cite{harary1966reconstruction} showed how to recover a tree using only the cards that are subtrees, Bondy \cite{bondy1969kelly} showed that only the cards where peripheral vertices have been removed are needed and Manvel \cite{manvel1970reconstruction} subsequently showed that the set (as opposed to the multiset) of cards that are trees suffices (except in four cases). Lauri \cite{lauri1983proof} also showed that trees with at least three cutvertices can be reconstructed (amongst all graphs) from the cards corresponding to removing a cutvertex. Indeed, Myrvold \cite{myrvold1990ally} proved that only three carefully chosen cards are needed to reconstruct a tree when $n \geq 5$. Related problems have also been investigated extensively for infinite trees (see, for example \cite{HSS72,BH74,T78,A81,N91,A94}), and it was recently shown by Bowler, Erde, Heinig, Lehner and Pitz that there are non-reconstructible locally finite trees \cite{BEHL17}.

Returning to the small cards setting, we have already mentioned Manvel's result in~\cite{Manvel74} that the class of connected graphs is recognisable from the $(n-2)$-deck for $n\geq 6$. Extending this, Kostochka, Nahvi, West, and Zirlin \cite{KNWZdc} showed that the connectedness of a graph on $n\geq 7$ vertices is determined by $\mathcal{D}_{n-3}(G)$. As shown by Spinoza and West \cite{SW19}, if we take $G_1=P_{n}$ (the path on $n$ vertices) and $G_2=C_{\lceil n/2 \rceil + 1}\sqcup P_{\lfloor n/2 \rfloor -1}$ the disjoint union of a cycle and a path, we find $\mathcal{D}_{\ell}(G_1)=\mathcal{D}_{\ell}(G_2)$ for all $\ell \leq \lfloor n/2 \rfloor$. However, $G_1$ is connected and $G_2$ is not. In light of this construction, Spinoza and West believe that for $n \geq 6$ and $\ell \geq \lfloor n/2 \rfloor+1$, the connectedness of an $n$-vertex graph $G$ is determined by $\mathcal{D}_\ell(G)$. This threshold would be sharp.

Spinoza and West proved in \cite{SW19} that connectedness can be recognised from $\mathcal{D}_\ell (G)$ provided 
\[
n-\ell \leq (1 + o(1))\sqrt{ \frac{2\log n}{\log (\log n)}}.
\]
We significantly improve this bound to allow a linear gap between $n$ and $\ell$.
\begin{restatable}{theorem}{connectednesssmallcards}
\label{thm:connectedness_small_cards}
The connectedness of an $n$-vertex graph $G$ can be recognised from $\mathcal{D}_\ell(G)$ provided $\ell \geq  9n/10$.\end{restatable}

By Theorem \ref{thm:connectedness_small_cards} (and the fact that we can reconstruct the number of edges), we can recognise trees from the $\ell$-deck when $\ell\geq 9n/10$.
In order to prove Theorem \ref{thm:trees}, we need a slightly stronger bound.

\begin{restatable}{theorem}{treerecog}
\label{thm:treerecog}
For $\ell \geq (2n + 4)/3$, the class of trees on $n$ vertices is recognisable from the $\ell$-deck.
\end{restatable}

As we were completing this paper, Kostochka, Nahvi, West and Zirlin \cite{Kostochkaea21} independently announced a similar result to Theorem \ref{thm:treerecog}. In fact, they proved that one can recognise if a graph is acyclic from the $\ell$-deck when $\ell \geq \floor{n/2}+1$, which also verifies the believed bound for reconstructing connectedness in the special case of forests. This has the particularly nice consequence that trees can be recognised from their $\ell$-deck, and so Conjecture~\ref{conj:trees} is equivalent to the reconstruction of trees amongst general graphs. Since our proof of Theorem~\ref{thm:treerecog} is short and already (more than) sufficient for our purpose of reconstructing trees, we have retained it to keep the proof of our main result self-contained.

The proof of Theorem \ref{thm:connectedness_small_cards} relies on an algebraic result (Lemma \ref{lem:moments}) which we also apply to reconstructing degree sequences. The story in the literature here is similar to that for connectedness. Chernyak \cite{Cher82} showed that the degree sequence of an $n$-vertex graph can be reconstructed from its $(n-2)$-deck for $n\geq 6$, and this was later extended by Kostochka, Nahvi, West, and Zirlin \cite{KNWZdc} to the $(n-3)$-deck for $n\geq 7$. 
The best known asymptotic result is due to Taylor \cite{T90}, 
and implies that the degree sequence of a graph $G$ on $n$ vertices can be reconstructed from $\mathcal{D}_\ell(G)$ where $\ell \sim (1 - 1/e) n$.
Our improved bound is as follows.

\begin{restatable}{theorem}{degreesmallcards}
\label{thm:degree_small_cards}
The degree sequence of an $n$-vertex graph $G$ can reconstructed from $\mathcal{D}_\ell(G)$ for any $\ell \geq\sqrt{2n\log(2n)}$.
\end{restatable}

In Section~\ref{sec:prelim}, we give $\ell$-deck versions of both Kelly's Lemma \cite{K57} and a result on counting maximal subgraphs by Greenwell and Hemminger \cite{GH73}, as well as an algebraic result of Borwein and Ingalls \cite{BI94} bounding the number of moments shared by two distinct sequences. These are used to deduce Theorem \ref{thm:degree_small_cards} (Section~\ref{sec:degree}) and Theorem \ref{thm:connectedness_small_cards} (Section~\ref{sec:conn}). Our main result on reconstructing trees, Theorem~\ref{thm:trees}, is proved in Section~\ref{sec:trees} which also contains a new counting tool for reconstruction that may be of independent interest as well as the proof of Theorem~\ref{thm:treerecog}. We conclude with some further discussion in Section~\ref{sec:end}.

\section{Preliminaries}
\label{sec:prelim}
This paper makes extensive use of three key results which we give in this section.

\subsection{Kelly's Lemma}

Let $\midtilde{n}_H(G)$ and $n_H(G)$ denote the number of subgraphs and induced subgraphs of $G$ isomorphic to $H$ respectively. We will reserve the word \emph{copy} of $H$ for an induced subgraph isomorphic to $H$, and say an \emph{instance} of $H$ to mean not necessarily induced.

In the classical graph reconstruction problem, Kelly's Lemma states that we can reconstruct $n_H(G)$ and $\midtilde{n}_H(G)$ provided $|V(H)| < |V(G)|$, and there are many variants of the lemma for other reconstruction problems~(see \cite{B91}). We use the following variant. 

\begin{lemma}
\label{lem:Kelly}
Let $\ell\in \N$ and let $H$ be a graph on at most $\ell$ vertices. For any graph $G$, the multiset of $\ell$-vertex induced subgraphs of $G$ determines both the number of subgraphs of $G$ that are isomorphic to $H$ and the number of induced subgraphs that are isomorphic to $H$.
\end{lemma}
In particular,  Kelly's Lemma means that $\mathcal{D}_{\ell'}(G)$ can be reconstructed from $\mathcal{D}_\ell(G)$ for all $\ell'\leq \ell$.

Despite its great usefulness, the proof of Kelly's Lemma requires nothing more than elementary counting. Suppose we count the number of copies of $H$ in each of the $\ell$-cards of $G$, and take the sum over all cards. Each copy of $H$ in $G$ will be counted exactly $\binom{n-|V(H)|}{\ell-|V(H)|}$ times toward this total. Hence, we can reconstruct the number $n_H(G)$ of copies of $H$ in $G$ from the $\ell$-deck as
\[
n_H(G)=\binom{n-|V(H)|}{\ell-|V(H)|}^{-1}\sum_{C\in \mathcal{D}_\ell(G)}n_{H}(C).
\]
The same argument applies with instances rather than copies.
Foreshadowing later usage, we remark that Kelly's Lemma only requires the subset of the deck consisting of the cards which contain at least one copy of the fixed graph $H$.

\subsection{Counting maximal subgraphs}\label{subsec:maxcount}
Given a class of graphs $\mathcal{F}$, a subgraph $F'$ of some graph $G$ is said to be an \emph{$\mathcal{F}$-subgraph} if $F'$ is isomorphic to some $F \in \mathcal{F}$, and is a \emph{maximal $\mathcal{F}$-subgraph} if the subgraph $F'$ cannot be extended to a larger $\mathcal{F}$-subgraph, that is, there does not exist an $\mathcal{F}$-subgraph $F''$ of $G$ such that $V(F') \subsetneq V(F'')$.

Let $m(F,G)$ denote the number of maximal $\mathcal{F}$-subgraphs in $G$ which are isomorphic to $F$. We give a slight variation of a classical ``Counting Theorem" due to Bondy and Hemminger \cite{BH77} (see also the statement of Greenwell and Hemminger \cite{GH73}) which reconstructs $m(F,G)$ from the $\ell$-deck.

\begin{lemma}
\label{lem:BH_counting}
Let $n\in \N$, let $\ell\in [n-1]$ and let $\mathcal{G}$ be a class of $n$-vertex graphs. Let $\mathcal{F}$ be a class of graphs such that for any $G\in \mathcal{G}$ and for any $\mathcal{F}$-subgraph $F$ of $G$,
\begin{enumerate}[label=(\roman*)]
\itemsep=0mm
    \item $|V(F)|\leq \ell$;
    \item $F$ is contained in a unique maximal $\mathcal{F}$-subgraph of $G$.
\end{enumerate}
Then for all $F\in \mathcal{F}$ and $G\in \mathcal{G}$, we can reconstruct $m(F,G)$ from the collection of cards in the $\ell$-deck that contain an $\mathcal{F}$-subgraph.
\end{lemma}

The following proof is essentially that of Bondy and Hemminger \cite{BH77}, only with a few additional observations used to accommodate our slight changes to the assumptions. 

\begin{proof}
Define an \textit{$(F,G)$-chain} of length $k$ to be a sequence $(X_0,\dots,X_k)$ of $\mathcal{F}$-subgraphs of $G$ such that
\[
F\cong X_0\subsetneq X_1\subsetneq \dots \subsetneq X_k \subsetneq G.
\]
The \textit{rank} of $F$ in $G$ is the length of a longest $(F,G)$-chain, and two chains are called \emph{isomorphic} if they have the same length and the corresponding terms are isomorphic. Following Bondy and Hemminger's argument, we first show that
\begin{equation}\label{eq:bhmaximal}
    m(F,G) = \sum_{k=0}^{\rank{F}} \sum
    (-1)^k \widetilde{n}_{F}(X_1) \widetilde{n}_{X_1}(X_2) \cdots \widetilde{n}_{X_{k-1}}(X_{k}) \widetilde{n}_{X_k}(G)
\end{equation}
where 
the second summation is over
all non-isomorphic $(F,G)$-chains of length $k$. 
When $\rank{F}=0$, we have $m(F,G)=\widetilde{n}_F(G)$. Let $\rank F=r$, and suppose that \eqref{eq:bhmaximal} holds for all graphs $F\in \mathcal{F}$ with rank less than $r$. The second assumption states that every copy of $F$ has a unique maximal extension $X$,
which implies that
\[
\widetilde{n}_F(G)=\sum_X \widetilde{n}_F(X)m(X,G),
\]
where the sum is over all non-isomorphic $\mathcal{F}$-subgraphs $X$ of $G$. This gives the expression
\[
m(F,G) = \widetilde{n}_F(G) - \sum_{X\not\cong F} \widetilde{n}_F(X)m(X,G).
\]
In the summation, we can restrict to $X$ for which $\widetilde{n}_F(X) >0$. Such a graph $X$ has rank at most $r-1$, so we may apply the induction hypothesis to rewrite each $m(X,G)$-term into a double sum.
The resulting triple sum can be simplified to obtain \eqref{eq:bhmaximal}. 

It now suffices to show that the RHS of \eqref{eq:bhmaximal} is reconstructible. To see this, we note that the inner summation is over $(F,G)$-chains for which $X_k$ has size at most $\ell$ (since $X_k$ is an $\mathcal{F}$-subgraph and by condition (i)), and so all such chains can be seen on cards. The remaining terms can be reconstructed by Kelly's Lemma (again using (i)), and this only requires the cards from $\mathcal{D}_\ell(G)$ that contain an $\mathcal{F}$-subgraph.
\end{proof}

\subsection{Shared moments of sequences}
We will need a bound on the maximum number of shared moments that two sequences $\alpha, \beta \in \{0,\dots, n\}^m$ can have. This result follows from the following theorem on the number of positive real roots of a polynomial. We use $\log$ to mean the natural logarithm here.

\begin{theorem}[Theorem A in \cite{BEK99}]
\label{thm:BEK}
Suppose that the complex polynomial
\[
p(z) := \sum_{j=0}^n a_j z^j
\]
has $k$ positive real roots (counted with multiplicity). Then
\[
k^2 \leq 2n \log\left(\frac{|a_0| + |a_1| + \dots + |a_n|}{\sqrt{|a_0a_n|}}\right).
\]
\end{theorem}
This theorem is attributed to Schmidt, but the first published proof is due to Schur and a series of simplifications have followed (see~\cite{BEK99}). The specific application that we require was given by Borwein and Ingalls \cite[Proposition 1]{BI94}. 
We shall use the following formulation which is tailored to our purposes.
\begin{lemma}
\label{lem:moments}
Let $\alpha,\beta \in \{0,\dots,n\}^m$ be two sequences that are not related to each other by a permutation. If
\begin{equation}
\label{eq:binoms}
\binom{\alpha_1}j+\dots +\binom{\alpha_m}j = \binom{\beta_1}j+\dots +\binom{\beta_m}j~~\text{ for all }j\in \{0,\dots,\ell\},
\end{equation}
then $\ell+1\leq \sqrt{2n\log(2m)}$.
\end{lemma}
\begin{proof}
Since $\alpha_i,\beta_j\in \{0,\dots,n\}$ for all $i,j\in [m]$, 
\begin{equation}
\label{eqn:pab}
    p_{\alpha,\beta}(x) :=\sum_{i=1}^m x^{\alpha_i}-\sum_{i=1}^m x^{\beta_i}
\end{equation}
is a polynomial of degree at most $n$. For $c \in \mathbb{C}$, let $\text{mult}_c(p_{\alpha,\beta})$ denote the multiplicity of the root at $c$, or 0 if $c$ is not a root of $p_{\alpha,\beta}$. We will show that $\ell + 1\leq \text{mult}_1(p_{\alpha,\beta})\leq  \sqrt{2n\log(2m)}$.

Since $\alpha$ and $\beta$ are not related by a permutation, the polynomial $p_{\alpha,\beta}$ is non-zero. We may write (with $r=\text{mult}_0(p_{\alpha,\beta})$)
\[
p_{\alpha,\beta}(x)=x^r \left(\sum_{j=0}^{n'}a_jx^j\right)
\]
where $a_0$ and $a_{n'}$ are non-zero and $n'\leq n$ . The coefficients are all integral, so $\sqrt{|a_0a_{n'}|}\geq 1$. Moreover, from the definition of the polynomial in (\ref{eqn:pab}) there are at most $2m$ contributions of $\pm1$ to the coefficients, so we have $\sum_{i=0}^{n'} |a_i|\leq 2m$.
    
By Theorem \ref{thm:BEK}, the number of positive real roots of $\sum_{j=0}^{n'}a_jx^j$ is at most
\[
\sqrt{2n' \log\left(\frac{|a_0| + |a_1| + \dots + |a_{n'}|}{\sqrt{|a_0a_{n'}|}}\right)}\leq \sqrt{2n \log(2m)}
\]
and in particular, $\text{mult}_1(p_{\alpha,\beta})\leq \sqrt{2n \log(2m)}$. On the other hand, for all $j\in \{0,\dots,\ell\}$, equation (\ref{eq:binoms}) shows that
    \[
   \left|\left(\frac{d}{dx^j}\left[\sum_{i=1}^m x^{\alpha_i}-\sum_{i=1}^m x^{\beta_i}\right]\right)\right|_{x=1}= \sum_{i=1}^m j!\binom{\alpha_i}j-\sum_{i=1}^mj!\binom{\beta_i}j=0.
    \]
    Hence, $\ell + 1 \leq \text{mult}_1(p_{\alpha,\beta})$, and $\ell+1\leq \sqrt{2n\log(2m)}$ as desired.
\end{proof}
Condition (\ref{eq:binoms}) is equivalent to the condition that the first $\ell$ moments of $\alpha$ and $\beta$ agree. To see this, observe that $\{x^i:i\in \{0,\dots,\ell\}\}$ and $\{\binom{x}i:i\in \{0,\dots,\ell\}\}$ both form a basis for the polynomials of degree at most $\ell$. When $\alpha,\beta$ can be arbitrary integer sequences (instead of taking values in $\{0,\dots,n\}$) this variant is sometimes called the Prouhet-Tarry-Escott problem, and sequences are known with the first $\Omega(\sqrt{m})$ moments in common (see \cite[Proposition 3]{BI94} for a simple counting argument).

\section{Reconstructing the degree sequence}
\label{sec:degree}
The tools of the preceding section allow us to prove that the degree sequence of an $n$-vertex graph $G$ can be reconstructed from the $\ell$-deck of $G$ whenever $\ell\geq \sqrt{2n\log(2n)}$. The proof is essentially identical to that given by Taylor \cite{T90}, except for the use of the stronger bounds provided by Lemma~\ref{lem:moments}.

\degreesmallcards*

\begin{proof}
Let $G$ have vertices $v_1, \dots, v_n$, and let $\ell\geq \sqrt{2n\log(2n)}$ be an integer. 
By Lemma~\ref{lem:Kelly}, we can reconstruct the number of subgraphs of $G$ isomorphic to the star $K_{1,j}$ for all $j\in \{2,\dots,\ell-1\}$. Since vertex $v$ lies at the centre of $\binom{d(v)}{j}$ copies of $K_{1,j}$, we can compute the quantity
\[
\midtilde{n}_{K_{1,j}}(G)=\sum_{v\in V(G)} \binom{d(v)}{j}
\]
from the $\ell$-deck. We can also reconstruct
\[
\sum_{v\in V(G)} \binom{d(v)}{0}=n \text{ and }\sum_{v\in V(G)} \binom{d(v)}{1}=2\cdot e(G)
\]
from the $2$-deck.
Write $\alpha_i=d(v_i)$ for $i\in [n]$ where we may assume $d(v_1)\leq \dots \leq d(v_n)$. Suppose, for a contradiction, that a different degree sequence $\beta_1\leq \dots \leq \beta_n$ gives the same counts. Then, for $j \in \{0, \dots, \ell -1\}$,
\[
\sum_{i=1}^n \binom{\alpha_i}j=\sum_{i=1}^n \binom{\beta_i}j.
\]
Since $\alpha, \beta\in\{0,\dots,n-1\}^n$ are not permutations of each other, Lemma \ref{lem:moments} applies to show $\ell\leq \sqrt{2(n-1)\log(2n)}$ as desired.
\end{proof}

\section{Recognising connectedness}
\label{sec:conn}
In this section, we prove our theorem on reconstructing connectedness from the $\ell$-deck. Recall that throughout this paper, a copy $H'$ of $H$ in some graph $G$ refers to an induced subgraph of $G$ that is isomorphic to $H$. 

The main idea of the proof is that a graph $G$ has a connected component isomorphic to some graph $H$ on strictly less vertices than $G$, if and only if it has an induced subgraph isomorphic to $H$ `without any neighbours'. By a similar approach to the previous section, when $|V(H)|$ is small we can actually compute the entire `degree sequence', that is, for each $k$ we can find the number of induced copies of $H$ with $k$ `neighbours'. So we are done if $G$ has a small component. But if $G$ has no small components, then it is either connected or only has medium-sized components (in which case we recognize that it has no large connected subgraphs and we are done). 

\connectednesssmallcards*

\begin{proof}
Let $G$ be an $n$-vertex graph and let $\eps = 1/10$, so our assumption is that $\ell \geq 9n/10 = (1- \eps)n$. We begin by making an additional assumption on the size of $n$; it was shown by Kostochka, Nahvi, West, and Zirlin  \cite{KNWZdc} that the connectedness of a graph can be recognised from the $(n-3)$-deck for $n \geq 7$, so we can assume that $n \geq 39$.

Using Lemma~\ref{lem:Kelly} we can count the number of connected subgraphs of $G$ on $\ell$ vertices. If there are no such subgraphs, the graph must be disconnected and we are done. We may therefore assume that $G$ is either connected, or its largest component has order at least $\ell$. In particular, if $G$ is not connected then it has a component of order at most $n-\ell$.

We will reconstruct all components that have at most $n - \ell $ vertices from the $\ell$-deck. Let $H$ be a connected graph with $h$ vertices, where $1 \leq h \leq \eps n$.
Since $h\leq \ell$, we may compute $n_H(G)$ from the $\ell$-deck by Lemma \ref{lem:Kelly}.
Suppose $m=n_H(G)>0$. Write $H_1,\dots,H_m$ for the induced copies of $H$ in $G$, and define the \emph{neighbourhood} of $H_i$ by \[\Gamma(H_i) = \{v \in V(G)\setminus V(H_i): vu \in E(G) \text{ for some }u\in H_i\}.\]
Define the \emph{degree} of $H_i$ to be $|\Gamma(H_i)|$, and denote it by $\alpha_i$. Note that $G$ has a component isomorphic to $H$ if and only if $\alpha_i=0$ for some $i\in [m]$. Thus, reconstructing the sequence $(\alpha_1,\dots,\alpha_m)\in \{0,\dots, n-h\}^m$ determines the number of components isomorphic to $H$. 

We now show that we can reconstruct $(\alpha_1,\dots,\alpha_m)$ up to permutation. 
Since $1\leq h\leq \eps n$ and $m\leq \binom{n}h\leq \left(\frac{en}{h}\right)^h$, we have 
\begin{align*}
    \sqrt{2(n-h)\log(2m)} &\leq \sqrt{2(n-h)h\log(en/h) + 2n \log(2)}\\
    &\leq  n\sqrt{2(1-\eps)\eps\log(e/\eps) + 2\log(2)/n},
\end{align*}
where we also have that $(n-h)h\log(en/h)$ is increasing in $h$ within the given range. Hence by Lemma \ref{lem:moments}, it suffices to show that we can reconstruct 
\begin{equation}
\label{eq:binoms_conn}
\sum_{i=1}^m\binom{\alpha_i}j \text{ for all integers }0\leq j\leq N,
\end{equation}
where $N = n\sqrt{2(1-\eps)\eps\log(e/\eps) + 2\log(2)/n}$.

Let $P$ denote the set of pairs of vertex sets $(A,B)$ where $A\subseteq B\subseteq V(G)$, $G[A]\cong H$, $|B|=|A|+j$ and $A$ is \textit{dominating} in $G[B]$ -- that is, each vertex in $B\setminus A$ is adjacent to some vertex in $A$. Each $(A,B)\in P$ has some $i\in [m]$ for which $G[A]\cong H_i$ and $B$ is contained in the neighbourhood of $H_i$, so $|P| = \sum_{i = 1}^m \binom{\alpha_i}{j}$.

For $j \geq 0$, let $\mathcal{H}_j$ denote the set of $(h+j)$-vertex graphs that consist of $H$ along with $j$ additional vertices, all of which are adjacent to at least one vertex in the copy of $H$ (we include each isomorphism type once). If $(A,B)\in P$, then $B$ corresponds to some $H' \in \mathcal{H}_j$. 
By definition, there are $n_{H'}(G)$ vertex sets $B\subseteq V(G)$ with $G[B]\cong H'$. Since $\mathcal{H}_j$ and $H$ are known to us, for each $H'\in \mathcal{H}_j$ we can calculate the number $n(H,H')$ of dominating copies of $H$ in $H'$. Since
\[
\sum_{H'\in \mathcal{H}_j}n(H,H')n_{H'}(G)=|P|=\sum_{i=1}^m\binom{\alpha_i}j,
\]
it only remains to show that we can determine $n_{H'}(G)$ from the $\ell$-deck.

 We may use Lemma \ref{lem:Kelly} to reconstruct $n_{H'}(G)$ if $|H'|=h+j\leq \ell$. For $j\leq N$ and $n\geq 39$, we find that
\[
h+j\leq \eps n + N\leq n-\eps n\leq \ell,
\]
where the middle inequality follows from the fact that, using $\eps = 1/10$, we have
\begin{equation*}
\label{eqn:bound}
    \sqrt{2(1-\eps)\eps\log(e/\eps) + 2\log(2)/39} \leq
1 - 2 \eps.
\end{equation*}

This shows that we can reconstruct (\ref{eq:binoms_conn}), and hence the number of components isomorphic to $H$. In particular, doing so for every graph $H$ with at most $n-\ell$ vertices allows us to determine whether any component of $G$ has at most $n-\ell$ vertices, which we saw would hold if and only if $G$ is disconnected.
\end{proof}

We remark that the constant $9/10$ can be improved slightly in the proof above provided $n$ is large enough. Indeed, the proof holds for any $n$ and $\eps$ such that \[\sqrt{2(1-\eps)\eps\log(e/\eps) + 2\log(2)/n} \leq
1 - 2 \eps,\] and, for large enough $n$, we can take $\eps \approx 0.1069$.

\newpage
\section{Reconstructing trees}
\label{sec:trees}
We now work toward proving our main theorem on reconstructing trees, which we recall below.

\trees*

The proof of Theorem \ref{thm:trees} is spread across the following four subsections. First, we introduce a general technique for counting balls around a subgraph, which may be of independent interest. This strategy allows us to keep track of copies of fixed graphs in $T$ that have a specified distinguished subgraph, which is a crucial ingredient of our proofs. This is done in Section~\ref{sec:extcount}.

In Section~\ref{sec:subrecog}, we address the recognition problem and prove Theorem~\ref{thm:treerecog}.

The remaining parts contain the proof of reconstruction, which is split into two cases depending on whether or not the tree $T$ contains a path that is long relative to the order of the graph $n$ and the number $\ell$ of vertices on each card.
Let the \emph{length} of a path $P$ be the number of edges in $P$, or equivalently $|V(P)|- 1$. The \emph{diameter} of a graph $G$ is the maximum distance between two vertices in $G$, and for a tree $T$ this is the same as the length of a longest path. When the diameter is less than about $\ell-2n/3$, we can apply an argument based on reconstructing branches off the centre. For trees with diameter higher than this (in fact there is some overlap between cases), we will split the tree into two parts by removing a central edge, and then recognising these parts and how to glue them back together.

Having recognised that every reconstruction of the deck is a tree, the high diameter case is handled by the following lemma which we prove in Section \ref{sec:high_diam}. 

\begin{restatable}{lemma}{highdiam}
\label{lem:high_diam}
Let $\ell,k\in [n]$ with $k > 4\sqrt{\ell} + 2(n-\ell)$. If $T$ is an $n$-vertex tree with diameter $k -1$, then $T$ can be reconstructed amongst connected graphs from its $\ell$-deck provided $\ell \geq \frac{2n}{3} + \frac{4}{9} \sqrt{6n + 7} + \frac{11}{9}$. 
\end{restatable}

If $T$ has low diameter, then we instead use the following lemma which we prove in Section \ref{subsec:low_diam}.

\begin{restatable}{lemma}{lowdiam}
\label{lem:low_diam}
Let $\ell,k\in [n]$ with $k < \ell - \frac{2n+1}{3}$. If $T$ is an $n$-vertex tree with diameter $k -1$, then $T$ is reconstructible from its $\ell$-deck.
\end{restatable}

The proof of Theorem~\ref{thm:trees} then amounts to verifying that the assumptions are sufficient for recognition, and that our definitions of high and low diameter together cover the full range. 
The latter calculation is the source of the threshold on card size in the statement of Theorem~\ref{thm:trees}. 

\begin{proof}[Proof of Theorem~\ref{thm:trees}]
Let $k$ be the number of vertices in the longest path in $T$.
The conditions on $\ell$ and $n$ imply that $\ell \geq \frac{2n}{3} + \frac{4}{9} \sqrt{6n + 7} + \frac{11}{9}$. This allows us to recognise that $T$ is a tree by Theorem~\ref{thm:treerecog}, and moreover that $T$ is reconstructible by Lemma~\ref{lem:high_diam} when $k> 4\sqrt{\ell} + 2(n-\ell)$. 
We show that the remaining $k$ satisfy the condition in Lemma~\ref{lem:low_diam}. 
It suffices to verify that $n-\ell < \frac{n-3k-1}{3}$. The right hand side is decreasing in $k$, and now $k \leq 4\sqrt{\ell} + 2(n-\ell)$, so Lemma~\ref{lem:low_diam} applies provided 
    \[n-\ell< \frac{n-12\sqrt{\ell}-6(n-\ell)-1}{3}\]
    which is equivalent to our assumed condition
    \[\ell > \frac{8n}{9} + \frac{4}{9}\sqrt{8 n + 5} +1. \qedhere\]
\end{proof}

\subsection{Counting extensions}\label{sec:extcount}
Given a graph $H$, we define an \emph{$H$-extension} to be a pair $H_{\textrm{ext}}=(H^+,A)$ where $H^+$ is a graph and $A\subseteq V(H^+)$ is a subset of vertices with $H^+[A]\cong H$. The idea is that $H^+$ may contain multiple copies of $H$ as induced subgraphs, so we are picking out one in particular. The \textit{order} of $H_{\textrm{ext}}=(H^+,A)$ is $|H_{\textrm{ext}}|=|V(H^+)|$.

We will usually work with $H$-extensions in a setting where $H$ is an induced subgraph of an ambient graph $G$, and in this case a natural family of $H$-extensions can be obtained by considering neighbourhoods. Specifically, for $d\in \N$, the \emph{(closed) $d$-ball} of an induced subgraph $H$ of a graph $G$ is
\[
B_{d}(H,G)=G[\{v\in V(G): d_G(v,H)\leq d\}],
\]
the subgraph induced by the set of vertices of distance at most $d$ from $H$ including the vertices of $H$ itself. It is useful to view the $d$-ball of $H$ as the $H$-extension $(B_d(H,G),V(H))$. 

Two $H$-extensions $(G_1,A_1)$ and $(G_2,A_2)$ are \textit{isomorphic} if there is a graph isomorphism $\varphi:G_1\to G_2$ with $\varphi(A_1)=A_2$. Let $m_d(\Hext,G)$ be the number of copies of $H$ in $G$ whose $d$-ball is isomorphic (as an $H$-extension) to $\Hext$. In addition, we say that an $H$-extension $(H^+,A)$ is a \emph{sub-$H$-extension} of $(H^{++},B)$ if $H^+$ is an induced subgraph of $H^{++}$ and $A=B$. 

Our key counting result for extensions states that it is possible to reconstruct $m_d(\Hext,G)$ from the $\ell$-deck provided the $d$-balls of all copies of $H$ are small enough to appear on the cards. 

\begin{lemma}
\label{lem:maximal_IE}
Let $\ell,d\in \N$ and let $G$ be a graph on at least $\ell+1$ vertices. Let $H$ be a graph on at most $\ell-1$ vertices. From the $\ell$-deck of $G$, it is possible to recognise whether the $d$-ball of every induced copy of $H$ in $G$ has fewer than $\ell$ vertices, and if this is the case, the quantity $m_d(\Hext,G)$ is determined by the $\ell$-deck for any $H$-extension $\Hext$.
\end{lemma}

\begin{proof}
Let $\mathcal{H}$ denote the set of graphs $H^+$ such that $|V(H^+)|\leq \ell$ and there is a copy $H'$ of $H$ in $H^+$ in which all the vertices of $H^+$ are at distance (in $H^+$) at most $d$ from $H'$. These represent all possible $d$-balls of $H$ with at most $\ell$ vertices, and the ones that appear in $G$ will be a subset of these.
Note that it is not necessary (nor guaranteed) that all copies of $H$ in $H^+$ satisfy the above distance condition, rather only that there is at least one such copy.

For any $H^+\in \mathcal{H}$, we can reconstruct $n_{H^+}(G)$ from the $\ell$-deck using Lemma~\ref{lem:Kelly}.
The $d$-balls of every induced copy of $H$ have fewer than $\ell$ vertices if and only if the $n_{H^+}(G) = 0$ for every $H^+ \in \mathcal{H}$ with $|H^+| = \ell$, and we can tell if this is the case.
Suppose that indeed the $d$-balls around every induced copy of $H$ have fewer than $\ell$ vertices and set 
\[
k=\max\{|V(H^+)|:H^+\in \mathcal{H},~n_{H^+}(G)>0\}.
\]
For a fixed $H^+\in \mathcal{H}$ with $|V(H^+)|=k$, we observe that every copy $H'$ of $H$ for which $B_d(H',H^+)\cong H^+$ also satisfies $B_d(H', G) \cong H^+$ by the maximality of $k$ and the definition of $\mathcal{H}$.

Let $\calHext$ denote the set of isomorphism classes of $H$-extensions $(H^+,A)$ with $H^+\in \mathcal{H}$.
By the preceding observation, if $\Hext=(H^+,A)\in \calHext$ with $|H^+|=k$, then the number of copies of $H$ whose $d$-balls are isomorphic to $\Hext$ is the number of copies of $H^+$ in $G$ times the number of copies of $H$ in $H^+$ whose $d$-ball in $H^+$ is isomorphic to $\Hext$ (as $H$-extensions). That is, 
\begin{equation}
\label{eq:easy_case}
m_d(\Hext,G)={n}_{H^+}(G) m_d(\Hext,H^+),
\end{equation}
Both of these quantities are reconstructible from the $\ell$-deck, so we are done in this case.

If $|V(H^+)| < k$, then the $d$-ball of $H$ may be strictly larger than $H^+$ and the formula (\ref{eq:easy_case}) does not apply. This can be corrected by subtracting the number of $H\subseteq H^+$ for which $H^+$ is not the $d$-neighbourhood of that copy of $H$ in $G$. To count these, we select each `maximal' $d$-neighbourhood in turn, and subtract one from the relevant count for each strictly smaller $H^+$ that it contains. Any leftover $H^+$ that have not been accounted for must then be maximal.

Explicitly, for $\Hext'\in\calHext$ distinct from $\Hext$, let $n(\Hext,\Hext')$ give the number of sub-$H$-extensions of $\Hext'$ isomorphic to $\Hext$. We claim that 
\[
m_d(\Hext,G)= {n}_{H^+}(G) m_d(\Hext,H^+)-\sum_{\substack{\Hext'\in \calHext \\|\Hext'|>|\Hext|}} n(\Hext,\Hext')m_d(\Hext',G).
\]
Note that when $|\Hext|=k$, the formula above agrees with (\ref{eq:easy_case}). The terms $m_d(\Hext,H^+)$ $n(\Hext,\Hext')$ and the domain of the summation are already known to us, and we can reconstruct 
$n_{H^+}(G)$ for all $H^+\in \mathcal{H}$ using Kelly's Lemma. Moreover, we may assume that we have reconstructed the terms $m_d(\Hext',H^+)$ for $|\Hext'|>|\Hext|$ by induction with base case $|\Hext|=k$, so verifying the formula will complete the proof.

The first term of the formula ${n}_{H^+}(G) m_d(\Hext,H^+)$ counts the number of pairs $(A,B) \subseteq V(G) \times V(G)$ such that
\begin{itemize}
\itemsep=0mm
    \item $G[B]$ is a copy of $H^+$ (contributing 1 to ${n}_{H^+}$),
    \item $A\subseteq B$, and $G[A]$ is a copy of $H$ and is counted by $m_d(\Hext,H^+)$ for a fixed copy of $H^+$, 
    \item $B$ is a subset of the $d$-ball around $A$ (i.e. $B \subseteq B_d(G[A],G)$).
\end{itemize}
Compared to $m_d(\Hext,G)$, the above term overcounts by 1 whenever $B \subsetneq B_d(G[A],G)$. Thus, it just remains to verify that the number of pairs with $B \neq B_d(G[A],G)$ is given by
$$\sum_{|\Hext'|>|\Hext|} n(\Hext,\Hext')m_d(\Hext',G).$$ To see that this is true, by definition the correction term counts triples $(A,B,C)$ with $A\subseteq B\subsetneq C \subseteq V(G)$ such that
\begin{itemize}
\itemsep=0mm
    \item $G[A]$ is a copy of $H$, 
    \item $G[B]$ is a copy of $H^+$  
    \item $G[C]\cong B_d(G[A],G)$.
\end{itemize} 
Each pair $(A,B)$ with $B\neq B_d(G[A],G)$ is in a unique such triple, namely with $C=V(B_d(G[A],G))$; if $B=B_d(G[A],G)$, then no suitable $C$ with $B\subsetneq C$ can be found. 
\end{proof}

As an aside, we mention that by setting $d=1$ and considering the $H$-extension $(H,V(H))$ in Lemma~\ref{lem:maximal_IE}, one can count the number of components isomorphic to $H$. 
\begin{corollary}
Let $H$ and $G$ be graphs with $|V(H)|\leq \ell-1$ and $n=|V(G)|$. If there is no copy of $H$ in $G$ for which $|B_1(H,G)| \geq \ell$, then we can reconstruct the number of components of $G$ isomorphic to $H$ from $\mathcal{D}_\ell(G)$.
\end{corollary}

\subsection{Recognising trees}\label{sec:subrecog}

This section contains the proof of Theorem~\ref{thm:treerecog}, which is an application of the extension-counting result established in the previous section.

\treerecog*

\begin{proof}
Let $G$ be a graph and suppose we are given $\mathcal{D}_\ell(G)$. By Kelly's Lemma (Lemma~\ref{lem:Kelly}), we can reconstruct the number $m$ of edges provided $\ell \geq 2$. We may suppose that $m=n-1$, otherwise we can already conclude that $G$ is not a tree. It suffices to show that we can determine whether $G$ contains a cycle, or equivalently to determine whether $G$ is connected.

If $G$ has a cycle of length at most $\ell$, then the entire cycle will appear on a card and we can conclude that $G$ is not a tree. We may therefore assume that every cycle in $G$ has length greater than $\ell$. If the graph does not contain a connected card, then the graph cannot be a tree, and so we may assume that there is a connected card and the largest components in $G$ have at least $\ell$ vertices each. Since $\ell \geq (2n + 4)/3$, there is only one component $A$ with at least $\ell$ vertices and the other components have at most $\ell-1$ vertices. 

Let $d = \ceil{\ell - n/2 - 1}$. For a vertex $x\in V(G)$, denote the $d$-ball around $x$ in $G$ by $B_d(x)$. Using Lemma \ref{lem:maximal_IE} with $H$ being the graph consisting of a single vertex, we find that either there is an $x\in V(G)$ with $d$-ball of order at least $\ell$ or we can reconstruct the collection of $d$-balls (with `distinguished' centres). 

Suppose firstly that there exists $x\in V(G)$ such that $|B_d(x)| \geq \ell$. We claim that then $G$ is a tree. Assume towards a contradiction that there is a cycle in $G$. Since this must have more than $\ell$ vertices, any cycle in $G$ must be contained in the largest component $A$ (the smaller components have order at most $\ell-1$). Let $C$ be a shortest cycle in $A$. Similarly, note that $x \in A$ since otherwise the $d$-ball around $x$ cannot have $\ell$ vertices. If $|B_d(x) \cap V(C)| \leq 2d+1$, then
\[
|B_d(x)| \leq n - | V(C) \setminus B_d(x)| \leq n - (\ell +1) + (2d+1) \leq \ell -1\]
by our choice of $d$. Thus, $B_d(x) \cap V(C)$ contains at least $2d + 2$ vertices. Choose two vertices $c_1,c_2\in B_d(x)\cap V(C)$ joined by a subpath $C'$ of $C$ (possibly $C'$ is a single edge) such that $C'$ does not contain any other vertex of $B_d(x)$.
Let $C''$ be the other path from $c_1$ to $c_2$ in $C$. This must contain at least $2d$ other vertices of $B_d(x)\cap C$, so $C''$ is a path of length at least $2d + 1$. However, there is also a path $P$ from $c_1$ to $c_2$ in the $d$-ball around $x$ of length at most $2d$, and this intersects $C'$ only at the endpoints $c_1$ and $c_2$. Replacing the path $C''$ with the path $P$ forms a cycle which is strictly shorter than $C$, giving a contradiction. Hence, $G$ cannot have any cycles and must be a tree.

We may now assume that we can reconstruct the collection of $d$-balls and will show how to recognise whether the graph is connected in this case.
In any component of order at most $n - \ell$, there must be some vertex $x$ such that the distance from $x$ to any vertex in the same component is at most $(n - \ell)/2$. By our choice of $\ell$ and $d$,
\[\frac{n - \ell}{2} \leq \ell - \frac{n}{2} - 2 \leq d - 1.
\]
Thus, if there is a component of order at most $n-\ell$ (which happens if and only if $G$ is not a tree), then there must be a $d$-ball with radius at most $d-1$. Conversely, if we discover such a $d$-ball, then we know that the graph is disconnected since the $d$-ball must form a component due to its radius, yet has at most $\ell-1$ vertices. Hence, $G$ is a tree if and only if all $d$-balls have radius $d$. This shows that we can recognise connectedness and completes the proof.
\end{proof}

\subsection{High diameter}
\label{sec:high_diam}
The main result in this section is Lemma \ref{lem:high_diam}, which states that a tree $T$ is reconstructible from its $\ell$-deck provided it contains a sufficiently long path. 

Removing an edge $e$ from a tree $T$ splits $T$ into two components, and our goal will be to recognise a pair of graphs $(R,R^c)$ which are the components left after removing an edge from $T$. However, it is not enough to know that $T$ is formed by connecting $R$ and $R^c$ with an edge, we also need to know which vertices the edge is connected to, and we will actually look for pairs for which we can also deduce this. 

We are specifically interested in induced subgraphs that are connected to the rest of the graph by a single edge, which leads us to consider copies of $R$ (and $R^c$) with this property.
For a graph $H$, let a \emph{leaf $H$-extension} be a pair $\Hext = (H^+, A)$ where 
\begin{itemize}
\itemsep=0mm
    \item $H^+$ is obtained by adding a single vertex connected by a single edge to a vertex of $H$, and
    \item $A\subset V(H^+)$ is such that $H^+[A] \cong H$.
\end{itemize} 
This is a special case of the extensions defined in Section~\ref{sec:extcount}.
We will refer to the additional edge added to $H$ to form $H^+$ as the \emph{extending edge}.
Note that if $R$ is a component of $T - e$, then the $1$-ball of $T[V(R)]$ in $T$ is a leaf $R$-extension, but there may be multiple (non-isomorphic) leaf $R$-extensions in $T$.

The extra edge in a leaf extension indicates where to glue, so we would be done if we could identify two leaf extensions $C=(C^+,V_C)$ and $D=(D^+,V_D)$ for which the vertex set of $G$ is the disjoint union of $V(C)$ and $V(D)$. We demonstrate in Lemma~\ref{lem:bridge} a case where this can be done from $\mathcal{D}_{\ell}(G)$ using counts of the relevant leaf extensions obtained by Lemma~\ref{lem:maximal_IE}.
Lemma~\ref{lem:bridge} is not specialised to trees with high diameter and the final step in proving Lemma~\ref{lem:high_diam} is showing that the lemma applies to trees with high diameter. 

We say an edge $e$ in a connected graph $G$ is a \textit{bridge} if the graph $G-e$ obtained from the removing the edge is disconnected. 
\begin{lemma}
\label{lem:bridge}
Let $G$ be a connected graph with a bridge $e$, and $R, R^c \subseteq G$ be the components of $G-e$. If $G$ has no induced subgraph $H$ isomorphic to $R$ or $R^c$ with $|V(B_1(H,G))|\geq \ell$, then $G$ is the only connected graph up to isomorphism with the deck $\mathcal{D}_\ell(G)$.
\end{lemma}

\begin{proof}
We prove the lemma by describing an algorithm that takes in the deck $\mathcal{D}_\ell(G)$ of a connected graph $G$, and either returns a connected graph, or a failure. We will show that if the algorithm returns a graph $G'$, it must be isomorphic to $G$.
This shows that such a $G$ is reconstructible since if $\mathcal{D}_\ell(G_1) = \mathcal{D}_\ell(G_2)$, then applying the algorithm to this shared deck will produce a single graph $G'$ for which $G' \cong G_1 \cong G_2$.
The condition in the hypothesis that $G$ has a suitable bridge $e$ is only used to show that the algorithm will definitely output a graph.

The idea of the procedure is to create a finite list of candidate graphs guaranteed to contain both components of $G-e$, and then test all pairs of such graphs glued together in every feasible way that could reconstruct $G$. This latter step is refined by using leaf extensions to indicate how these gluings occur. The key point is to show that we can identify when such a construction actually produces $G$ and then terminate. 

Given any connected graph $H$ on at most $\ell-1$ vertices and a deck $\mathcal{D}_\ell(G)$, we can check directly from the cards whether there is a copy $H'$ of $H$ in $G$ for which $|V(B_1(H',G))|\geq \ell$. Say that a graph $H$ is \emph{confined} if no such copy of it exists. Then for every confined connected graph $H$ and every leaf $H$-extension $\Hext$ of $H$, we can apply Lemma \ref{lem:maximal_IE} to reconstruct $m_1(\Hext,G)$. Recall that this is the number of copies of $H$ in $G$ whose $1$-ball in $G$ is obtained by adding a pendant vertex connected at a specified vertex, so a positive value would signal an extension that might correspond to a component of $G-e$ (with the extending edge corresponding to the bridge).  
To form our collection of candidates, let $\calHext$ denote the isomorphism classes of all leaf $H$-extensions $\Hext$ for which $m_1(\Hext,G)>0$ and $H$ is a confined connected graph.

We now loop over all pairs $(\Cext,\Dext)$ of elements from $\calHext$ for which $|\Cext|+|\Dext| = n + 2$ and $|\Cext| \leq |\Dext|$. Let $\Cext = (C^+, V_C)$ and $\Dext = (D^+, V_D)$ where $C = C^+[V_C]$ and $D= D^+[V_D]$ denote the corresponding labelled subgraphs. Let $N(\Cext, \Dext)$ be the number of copies of $C$ in $D$ whose $1$-ball in $D^+$ is a copy of $C^+$. That is, we count the copies of $C^+$ in $D^+$ where the extending edge of $D^+$ is either unused or is the extending edge of $C^+$. If $m_1(\Cext, G) \geq N(\Cext, \Dext) + 1$, then the algorithm outputs the graph $G'$ formed by taking disjoint copies of $C^+$ and $D^+$ and identifying their extending edges as given by these extensions. If $m_1(\Cext, G) < N(\Cext, \Dext) + 1$, we continue on to the next pair of elements of $\calHext$.  If we have checked every suitable pair of elements from $\calHext$ without outputting a graph, then we output a failure.

Let us first verify that if the algorithm returns a graph, it must be isomorphic to $G$. In fact, this is true for any connected graph. We will later use our assumptions on $G$ to argue that the algorithm does output a graph when the input is $\mathcal{D}_\ell(G)$, which shows that $G$ is reconstructible.

It is useful to highlight that every leaf extension $\Dext = (D^+, V_D)$ with $m_1(\Dext,G)>0$ has a unique \emph{partner} leaf extension, which we will denote by $\Dext^c = ((D^c)^+, V_{D^c})$, that produces a graph isomorphic to $G$ when joined with $\Dext$ as described above. Explicitly, if $e_D$ is the extending edge in $\Dext$, then $\Dext^c$ is given by taking $V_{D^c} = (V_D)^c$ and setting $(D^c)^+$ to be the complement of $D$ in $G$ together with $e_D$ as an additional edge. Indeed, this is true when $G$ is any connected graph.

We now argue that any output graph $G'$ is isomorphic to $G$. Let $(\Cext, \Dext)$ be a pair that produces $G'$. If $\Cext \cong \Dext^c$ as leaf extensions then $G\cong G'$ by definition of $\Dext^c$, so suppose this is not the case. It is enough to show that $m_1(\Cext, G) < N(\Cext, \Dext) + 1$, giving a contradiction to the fact that we terminated when considering $(\Cext,\Dext)$. We first claim that if a copy of $C$ contributes to $m_1(\Cext, G)$ (in the sense that $G[V_C]$ coincides with this copy for a leaf extension counted by $m_1(\Cext, G)$), then it cannot use the edge between $D$ and $D^c$. To see this, note that since $G$ is connected, both $D$ and $D^c$ are connected. We have assumed that $|V(C)| = |V(D^c)| \leq |V(D)|$ and $\Cext \not\cong \Dext^c$, so no copy of $C$ can completely contain either $D$ or $D^c$. This means that if a copy of $C$ were to use this edge, then its 1-ball would contain at least one vertex from $D$ and one from $D^c$ meaning it does not contribute to  $m_1(\Cext, G)$.

It now follows that 
\begin{equation}\label{eq:m1nobridge}
m_1(\Cext, G) = N(\Cext, \Dext) + N(\Cext, \Dext^c).
\end{equation}
 If $\Dext^c \not\cong \Cext$ as we have assumed, then $N(\Cext, \Dext^c) = 0$. Indeed, since no copy of $C$ uses the extending edge of $(D^c)^+$, this would have to be the extending edge of $C^+$ and we would have $\Cext \cong \Dext$. This leaves $m_1(\Cext, G) = N(\Cext, \Dext)$ which gives the desired contradiction to the supposition that we terminated when considering $(\Cext, \Dext)$, so $G'$ must be isomorphic to $G$. 

Finally, let us argue that the algorithm does terminate when the input is $\mathcal{D}_\ell(G)$. From \eqref{eq:m1nobridge} we easily see that $m_1(R_{\textrm{ext}}, G) = N(R_{\textrm{ext}}, R^c_{\textrm{ext}}) + 1$, where $R_{\textrm{ext}}$ and $R^c_{\textrm{ext}}$ are as defined earlier. Our assumption on $G$ guarantees that both of these are in$\calHext$, and we are guaranteed to have at least one pair among our candidates that will lead to termination.
\end{proof}

We remark that the only place where we used that the existence of the edge $e$ which splits $G$ into ``nice" components $R$ and $R^c$ was to ensure that the algorithm output a graph. One can try to use the algorithm to reconstruct graphs whenever the deck is known to correspond to a connected graph, and the algorithm will either output the graph, or a failure (in which case one needs a different approach). 
We now show that any tree with large enough diameter (depending on both $n$ and $\ell$) does have a bridge which splits the tree into ``nice" components, and so satisfies the required condition to be reconstructible amongst connected graphs.

\begin{proof}[Proof of Lemma \ref{lem:high_diam}]
Let $k,\ell\in [n]$ with $k > 4\sqrt{\ell} + 2(n-\ell)$ and $\ell \geq \frac{2n}{3} + \frac{4}{9} \sqrt{6n + 7} + \frac{11}{9}$. Let $T$ be a tree and suppose that a longest path in $T$ contains exactly $k$ vertices. We wish to show it has a suitable bridge satisfying the assumptions of Lemma \ref{lem:bridge} so that we can conclude that it is reconstructible from its $\ell$-deck amongst connected graphs.

Fix a longest path in $T$ with $k$ vertices. Create two rooted subtrees $R$ and $S=R^c$ by removing the central edge of the path if $k$ is even, or one of the two central edges if $k$ is odd (and rooting the subtrees at the vertex which had an incident edge removed). By Lemma \ref{lem:bridge}, if $T$ has no induced subgraph $H$ isomorphic to $R$ or $S$ with $|V(B_1(H,T))|\geq \ell$, then $T$ is reconstructible from $\mathcal{D}_\ell(T)$. We assume, in order to derive a contradiction, that $T$ contains a copy $S'$ of $S$ with $|V(B_1(S',T))| \geq \ell$.  Note that, since $R$ contains at least $n - \ell + 2\sqrt{\ell} - 1$ vertices, $S$ contains at most $\ell - 2$ vertices and the 1-ball of $S$ contains at most $\ell - 1$ vertices.

We will proceed by building a sequence of vertex-disjoint paths in $S$ to obtain a lower bound on the size of $S$, which leads to an upper bound on the maximum length of a path in $R$ and hence also an upper bound on $k$. Let us sketch the main idea. Since the 1-ball around $S'$ only extends one vertex further along paths not in $S'$, the existence of a long path in $T$ that is not in $S'$ would indicate that this $1$-ball misses lots of vertices and so cannot be too big. This means that $S'$ should contain a lot of our chosen longest path, and $S'$ should reach a long way into $R$. However, the long path in $S'$ that reaches into $R$ corresponds to a long path in $S$ (under the isomorphism that makes it a copy), and $S'$ must also contain many vertices from this path. These form another long path in $S$, and $S'$ must reach a long way down this path as well. Continuing this argument eventually forces $S$ to be so large that there are not enough vertices remaining to form a path of sufficient length in $R$ for our assumption on $k$ to be true.

Set $r=n-\ell$. Proceeding as laid out above, let $\varphi: S \to S'$ be an isomorphism, and let $P_0$ be a path in $R$ containing at least $(k-1)/2$ vertices which starts at the root of $R$. Consider the intersection of $S'$ with the path $P_0$. Since $V(S') \neq V(S)$, this intesection must be non-empty, and it must be connected since both $T$ and $S$ are trees, so $S'$ and $P_0$ intersect on a subpath $Q_0$. Moreover, the intersection of $B_1(S', T)$ and $P_0$ must also be a path with at most $|V(Q_0)| + 2$ vertices. 
This means that there are at least $|V(P_0)| - |V(Q_0)|- 2$ vertices on $P_0$ which are not in $B_1(S', T)$. 
We are assuming that $T$ has at most $r$ vertices which are not in $B_1(S',T)$, so $|V(Q_0)| \geq |V(P_0)| - r - 2$.

Now let $P_1$ be the path $\varphi^{-1}(V(Q_0))$ in $S$ and note that $P_1$ is vertex-disjoint from $P_0$ as $P_0$ is contained in $R$. Define $Q_1$ to be the intersection of $S'$ with $P_1$, which is again a path. Furthermore, the intersection of $B_1(S', T)$ and $P_1$ is also a path, this time with at most $|V(Q_1)| + 2$ vertices. The number of vertices of $P_1$ and $P_0$ which are not in $B_1(S',T)$ is at least $|V(P_0)| + |V(P_1)| - |V(Q_0)| - |V(Q_1)| - 4$, which gives the inequality $|V(Q_0)| + |V(Q_1)| \geq |V(P_0)| + |V(P_1)| - r- 4$. Since $|V(Q_0)| = |V(P_1)|$, this becomes $|V(Q_1)| \geq |V(P_0)| - r - 4$. 

We now continue to iteratively build our sequence of paths $P_i$, together with the sequence of subpaths $Q_i$ restricted to $S'$, as follows: given $P_i$ and $Q_i$, let $P_{i+1}:= \varphi^{-1}(V(Q_i))$ and set $Q_{i+1} = P_{i+1} \cap S'$. 
We first note that $P_{i+1}$ is disjoint from $P_0, \dots, P_i$. Indeed, since $P_0$ is contained in $R$,  $P_{i+1}$ cannot intersect $P_0$. If $P_{i+1}$ intersects a path $P_j$, then $Q_i$ must intersect $Q_{j-1}$ which in turn implies $P_{i}$ intersects $P_{j-1}$. Hence, the paths are disjoint by induction. By the finiteness of $T$, we must eventually reach a $j$ such that $|V(Q_{j-1})| = |V(P_j)| = 0$. At this point, we have disjoint paths $P_1, \dots, P_j$ in $S$ that satisfy $|V(P_i)| = |V(Q_{i-1})| \geq |V(P_0)| -r - 2i$ for all $i=1,\ldots,j$. In particular, setting $i=j$ to use the fact that $|V(P_j)| = 0$ shows that $j \geq (|V(P_0)|  -r)/2$. We may then calculate
\begin{align*}
    |V(S)| &\geq |V(P_1)| + \dotsb + |V(P_{j})|\\
    &\geq  \sum_{i=1}^{\floor{(|P_0|  -r)/2}} \left(|P_0| - r - 2i\right) \\
    &= (|P_0| - r) \floor{\frac{|P_0| - r}{2}} - 2 \binom{\floor{(|P_0| - r)/2} + 1}{2}\\
    &= \floor{\frac{|P_0| - r}{2}} \ceil{\frac{|P_0| - r - 2}{2}}\\
    &\geq \frac{(|P_0| - r)(|P_0| - r - 2)}{4},
\end{align*}
where we have used $|P_0|$ as shorthand for $|V(P_0)|$.

Since $|V(S)| \leq n - |V(P_0)|$, we must have $|V(P_0) | \leq \sqrt{4n - 4r + 1} + r - 1$ and $k \leq 2|V(P_0)|+1\leq 2 \sqrt{4n - 4r +1} + 2r - 1$. Finally, note that $2\sqrt{x + 1} - 1 \leq 2\sqrt{x}$ for all $x \geq 1$ to find $k \leq 4 \sqrt{\ell} + 2r$, a contradiction. The same argument shows that $T$ has no copy $R'$ of $R$ with $|V(B_1(R',T))|\geq \ell$. Hence, by Lemma \ref{lem:bridge} we can reconstruct $T$ from $\mathcal{D}_\ell(T)$.
\end{proof}

\subsection{Low diameter} 
\label{subsec:low_diam}

The purpose of this section is to prove Lemma~\ref{lem:low_diam}. Specifically, we will show that any tree $T$ with diameter $k-1$ can be reconstructed from its $\ell$-deck for any $\ell\in [n]$ such that $ n-\ell < \frac{n - 3k + 1}{3}$ if $k$ is odd or $n-\ell < \frac{n-3k -1}{3}$ if $k$ is even, which together imply the statement directly. These conditions are equivalent to odd $k < \ell - \frac{2n-1}{3}$ and even $k<\ell - \frac{2n+1}{3} $.
The reason for dependence on the parity is that, broadly, our strategy for reconstruction is to separately reconstruct branches of the tree emanating from its \emph{centre}: if $k$ is odd, the centre of $T$ is the vertex in the middle of each longest path, and if $k$ is even, the centre consists of the two middle vertices. The centre is unique, so in particular it does not depend on the choice of longest path. 

The case when $k$ is odd (so the diameter is even) is easier to work with, and the follow lemma provides a simple reduction that will allow us to proceed with this assumption.
 
\begin{lemma}\label{lem:parityreduction}
If all trees with $n+1$ vertices and even diameter are reconstructible from the cards in their $(\ell+1)$-decks that contain a longest path, then all trees with $n$ vertices and odd diameter are reconstructible from the cards in their $\ell$-decks that contain a longest path.
\end{lemma}
\begin{proof}
Let $T$ be any tree with $n$ vertices and odd diameter. This means that it has two middle vertices joined by one central edge. Let $T'$ be the tree obtained by subdividing the central edge of $T$, noting that it has $n+1$ vertices and even diameter. We can obtain the cards in the $(\ell+1)$-deck of $T'$ that contain a longest path by taking the cards in the $\ell$-deck of $T$ that contain a longest path and subdividing the central edge, and thus reconstruct $T'$ by assumption. It is then straightforward to recover $T$ by recognising the central vertex in $T'$ and smoothing out the vertex created by subdivision.
\end{proof}

Let us assume for the remainder of this section that $T$ is a tree with $n$ vertices, the number $k$ of vertices in a longest path in $T$ is odd, and $k < \ell - \frac{2n-1}{3}$.
This means that $k+1\leq \ell$ so we can reconstruct $k$ from the $\ell$-deck, which we shall use freely, and that $T$ has a unique central vertex. 

Given a vertex $u \in T$ with neighbours $v_1, v_2, \dots, v_a$, let the \emph{branches} at $u$ be the rooted subtrees $B_1,B_2, \dots, B_a$ where $B_i$ is the component of $T-u$ that contains $v_i$, rooted at $v_i$.
An \emph{end-rooted path} is a path rooted at an endvertex of the path. In this section, all longest paths $P_k$ will be rooted at the central vertex $c$, and are hence not end-rooted, whilst all of the shorter paths mentioned will be end-rooted. 
Given two rooted trees $T_1$ and $T_2$ with roots $u$ and $v$ respectively, let $T_1 \frown T_2$ denote the (unrooted) tree given by adding an edge between $u$ and $v$ (see Figure \ref{fig:def:frown}).
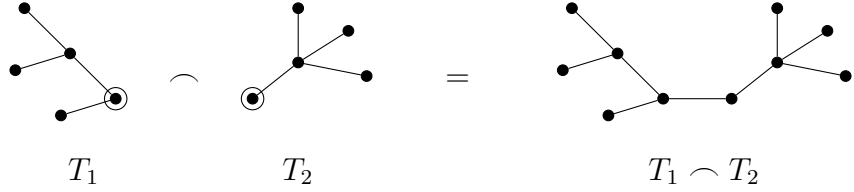
\begin{figure}
\centering
    \scalebox{0.85}{\begin{tikzpicture}[xscale=0.6, yscale=0.6]
        \draw node[style=vertex](0) at (-0.2,-0.37) {};
        \draw node[style=vertex](1) at (1,0) {};
        \draw node[style=vertex](2) at (0,1) {};
        \draw node[style=vertex](3) at (-1.2,0.63) {};
        \draw node[style=vertex](4) at (-1,2) {};
        
        \draw node[style=vertex](5) at (4,0) {};
        \draw node[style=vertex](6) at (5,0.8) {};
        \draw node[style=vertex](7) at (5,2) {};
        \draw node[style=vertex](8) at (6.1,1.5) {};
        \draw node[style=vertex](9) at (6.5,0.5) {};
        
        \draw node[style=vertex](10) at (11.8,-0.37) {};
        \draw node[style=vertex](11) at (13,0) {};
        \draw node[style=vertex](12) at (12,1) {};
        \draw node[style=vertex](13) at (10.8,0.63) {};
        \draw node[style=vertex](14) at (11,2) {};
        \draw node[style=vertex](15) at (15.5,0.8) {};
        \draw node[style=vertex](16) at (15.5,2) {};
        \draw node[style=vertex](17) at (16.6,1.5) {};
        \draw node[style=vertex](18) at (17,0.5) {};
        \draw node[style=vertex](19) at (14.5, 0) {};
        
        \node [fill=white] at (2.5, 0.4) {$\frown$};
        \node [fill=white] at (8.5, 0.4) {$=$};
        \draw [black] (1,0) circle (7pt);
        \draw [black] (4,0) circle (7pt);
        \node [fill=white] at (0.3, -1.6) {$T_1$};
         \node [fill=white] at (5, -1.6) {$T_2$};
         \node [fill=white] at (13.9, -1.6) {$T_1\frown T_2$};
        
        \draw (0) -- (1) -- (2) -- (3);
        \draw (2) -- (4);
        \draw (5) -- (6) -- (7);
        \draw (8) -- (6) -- (9);
        
        \draw (10) -- (11) -- (12) -- (13);
        \draw (12) -- (14);
       	\draw (11) -- (19) -- (15) -- (16);
	\draw (17) -- (15) -- (18);
    \end{tikzpicture}}
    \caption{An example of the tree grafting operation $T_1\frown T_2$ .}
    \label{fig:def:frown}
\end{figure}

By restricting our attention to the cards that have diameter $k - 1$, we may assume that we can always identify the centre of the graph. Our basic strategy is to reconstruct the branches at the centre separately, knowing that we can later join them together using the centre as a common point of reference. This can be done via a counting argument when all branches at the centre have at most $\ell-k$ vertices, but when one branch is `heavy' and contains many (at least $\ell-k$) of the vertices a slightly more finicky version of the argument is required to reconstruct this branch as it cannot be seen on a single card. It is possible to recognise these cases from the $\ell$-deck. We first address the simpler situation without heavy branches to illustrate the method.


\begin{lemma}\label{lemma:allsmall}
If $T$ is a tree with even diameter $k - 1$ for which every branch from the centre has fewer than $\ell-k$ vertices, then $T$ is reconstructible from the subset of the $\ell$-deck consisting only of cards that contain a copy of $P_k$.
\end{lemma}
\begin{proof}
Let $c$ be the central vertex of $T$, and let $\mathcal{B}=\{B_1, \dots, B_a\}$ be the branches at $c$ that we wish to reconstruct. 
If one of the branches at $c$ has at least $\ell-k$ vertices, then there must be a card containing a longest path with a branch of at least $\ell-k$ vertices (the branch and the path need not be disjoint, but their union contains at most $\ell$ vertices).
Thus we can recognise from the $\ell$-deck that all branches in $\mathcal{B}$ have fewer than $\ell-k$ vertices.

We first reconstruct all branches that are not end-rooted paths. For any fixed $B$ which is a rooted tree but not an end-rooted path, we will use Lemma~\ref{lem:BH_counting} to count each branch at $c$ isomorphic to $B$ once for every $P_k$ in $T$. Dividing this number, denoted $N_{B}$, by the number $n_{P_k}(T)$ of copies of $P_k$ in $T$ then tells us the multiplicity of $B$ in $T$ (which may be zero). Note that $n_{P_k}(T)$ can be determined by Kelly's Lemma as $k < \ell$, so it suffices to reconstruct $N_{B}$.

Following the preceding outline, fix $B$ to be any rooted tree that is not an end-rooted path. We will actually determine $N_{B}$ in two parts. Let $\pi_{B}$ be the number of pairs consisting of one copy $B'$ of $B$ that is a branch at $c$, and one copy $P_k'$ of a longest path that is disjoint from $B'$. Similarly, let $\tau_{B}$ count pairs $(B',P_k')$ where the copy $P_k'$ intersects $B'$. It is clear that $N_{B} = \pi_{B}+\tau_{B}$.

We begin with $\pi_{B}$. Let $\mathcal{G}$ be the family of all $n$-vertex trees with diameter $k - 1$ and where all branches from the centre have fewer than $\ell-k$ vertices. Let $\mathcal{F}$ be the family of graphs of the form $P_k \frown S$, where $S$ is a non-empty rooted tree with less than $\ell-k$ vertices that is not an end-rooted path and $P_k$ is rooted at its central vertex (see Figure~\ref{fig:families}). Fix $G \in \mathcal{G}$ and consider some $F \in \mathcal{F}$. If $F' = P_k' \frown S'$ is a copy of $F$ in $G$, then it is contained in a unique maximal $\mathcal{F}$-subgraph, namely $P_k'$ together with the unique branch $B'$ containing $S'$. Note that this would not be true if end-rooted paths were allowed, since the resulting $F'$ might then also be contained in a different maximal $\mathcal{F}$-subgraph $P''_k \frown B''$ where $S'$ is contained in the $P''_k$ and $B''$ is a branch that contains half of the original $P_k'$. Also, since $B'$ has fewer than $\ell-k$ vertices, these maximal elements have fewer than $\ell$ vertices and are therefore in $\mathcal{F}$. Thus, by Lemma~\ref{lem:BH_counting} we can reconstruct the number of $\mathcal{F}$-maximal copies of each $F$ in $G$ from $\mathcal{D}_\ell(G)$. This is non-zero for $F=P_k \frown S$ if and only if $\pi_{S}\neq 0$.

Now let $F = P_k \frown B$. Since $T\in \mathcal{G}$ and $F\in \mathcal{F}$, we may reconstruct the number of $\mathcal{F}$-maximal copies of $F$ in $T$ as above. This is precisely $\pi_{B}$. To see this, consider a particular copy $B'$ of $B$ that occurs as a branch and observe that $F$ occurs as a maximal $\mathcal{F}$-subgraph with this $B'$ as the copy of $B$ once for every longest path in the tree which avoids $B'$.

There is a similar argument to determine $\tau_{B}$. Keeping $\mathcal{G}$ as before, let $\mathcal{F}'$ be the family of graphs of the form $P_{(k-1)/2+1} \frown S$ where $S$ is a rooted tree which contains an end-rooted $P_{(k-1)/2}$, but is not itself an end-rooted path. Again, an element $F=P_{(k-1)/2+1} \frown S$ is $\mathcal{F}'$-maximal when $S$ is an entire branch, and for any $G \in \mathcal{G}$ and $F \in \mathcal{F}'$ we can reconstruct the number of $\mathcal{F}'$-maximal copies of each $F$ in $G$ by Lemma~\ref{lem:BH_counting}. This time there is at least one $\mathcal{F}'$-maximal copy of $F=P_{(k-1)/2+1} \frown S$ if and only if $G$ has a branch isomorphic to $S$ (although we do not need to use both directions explicitly).

Let $m_{F'}$ be the number of $\mathcal{F}'$-maximal copies of $F'=P_{(k-1)/2 + 1} \frown B$ in $T$, which we can reconstruct as argued above. A particular copy $B'$ of $B$ that occurs as a branch contributes 1 to $m_{F'}$ for each copy of $P_{(k-1)/2 + 1}$ that starts at the central vertex $c$ and is disjoint from $B'$. Thus, letting $n_{P^\bullet}(B)$ be the number of end-rooted copies of $P_{(k-1)/2 + 1}$ in $B'$ with roots coinciding (this is the same for any copy of $B$ and does not depend on the deck), one can construct all of the copies of longest paths that intersect $B'$ by gluing together one $P_{(k-1)/2 + 1}$ from inside $B'$ and one that is disjoint from it. Doing so for every copy of $B$ shows that we can reconstruct $\tau_{B}=m_{F'}\cdot n_{P^\bullet}(B)$.
The number of copies of $B$ that occur as a branch at $c$ can then be reconstructed as 
\[\frac{N_{B}}{n_{P_k}(T)}=\frac{\pi_{B}+\tau_{B}}{n_{P_k}(T)}.\]

It remains to determine the number of branches isomorphic to an end-rooted path $P_i$, which we do using the fact that we know all of the other branches not of this form. Starting with $j = (k-1)/2$, this being the maximum possible length of a path branch, we compare the number of copies of $P_{(k-1)/2+j + 1}$ in $T$ to the number of copies in the graph $\widetilde{T}$ obtained by gluing all of the known branches at a single vertex $c$. The former count can be obtained by Kelly's Lemma, and the latter by directly inspecting $\widetilde{T}$. If there are more copies in $T$ than in the current $\widetilde{T}$, then there must be at least one more end-rooted $P_j$ as a branch so we add one copy to our list of known branches. We then repeat this step with the same $j$ but a new $\widetilde{T}$ updated to include this new path branch. If the counts match, meaning all copies of $P_{(k-1)/2+j + 1}$ in $G$ are already present in $\widetilde{T}$, then reduce $j$ by 1 and continue iteratively until $j=0$. Note that it is important that we handle the different path lengths in this order. At this point, we have reconstructed all branches and the final $\widetilde{T}$ is exactly $T$.
\end{proof}

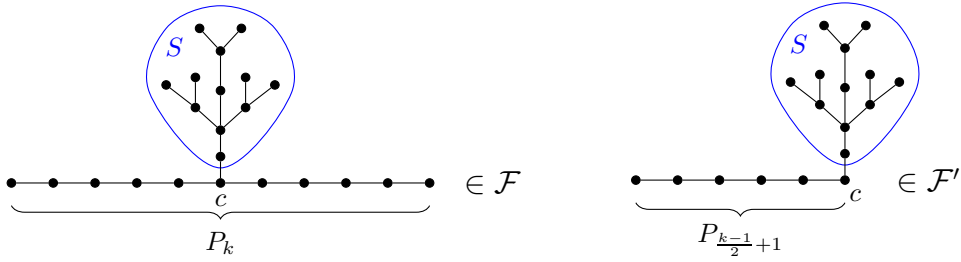
\begin{figure}
\tikzstyle{vertex}=[circle, draw, fill=black, inner sep=0pt, minimum width=3.2pt]
\centering
\begin{subfigure}{0.55 \textwidth}
    \centering
    \scalebox{0.85}{\begin{tikzpicture}[xscale=0.55, yscale=0.5]
         \node [fill=none] at (0, -0.45) {\footnotesize $c$};
         \node [fill=none] at (-1.1, 3.6) {\footnotesize \color{blue}$S$};
         \draw [blue] plot [smooth cycle, tension=0.6] coordinates {(0,0.4) (-1.7,2.4) (-1.3,4.1) (0, 4.7) (1.3,4.1) (1.7,2.4)};
    	\draw [decorate,decoration={brace,amplitude=5pt},xshift=0pt,yshift=0pt] (5,-0.6) -- (-5,-0.6) node [black,midway,xshift=0pt,yshift=-14pt]{\footnotesize $P_k$};
	\node [fill=none] at (6.5, 0) {$\in \mathcal{F}$};
    
        \draw node[style=vertex](0) at (-5,0) {};
        \draw node[style=vertex](1) at (-4,0) {};
        \draw node[style=vertex](2) at (-3,0) {};
        \draw node[style=vertex](3) at (-2,0) {};
        \draw node[style=vertex](4) at (-1,0) {};
        \draw node[style=vertex](c) at (0,0) {};
        \draw node[style=vertex](6) at (1,0) {};
        \draw node[style=vertex](7) at (2,0) {};
        \draw node[style=vertex](8) at (3,0) {};
        \draw node[style=vertex](9) at (4,0) {};
        \draw node[style=vertex](10) at (5,0) {};
        
        \draw node[style=vertex](11a) at (0, 0.7) {};
        \draw node[style=vertex](11) at (0,1.4) {};
        \draw node[style=vertex](12) at (-0.6,2) {};
        \draw node[style=vertex](13) at (0.6,2) {};
        \draw node[style=vertex](14a) at (0, 2.45) {};
        \draw node[style=vertex](14) at (0,3.5) {};
        \draw node[style=vertex](15) at (-1.3,2.6) {};
        \draw node[style=vertex](16) at (1.3,2.6) {};
        \draw node[style=vertex](17) at (-0.6,2.8) {};
        \draw node[style=vertex](18) at (0.6,2.8) {};
        \draw node[style=vertex](19) at (-0.5,4.1) {};
        \draw node[style=vertex](20) at (0.5,4.1) {};      
        
        \draw (0) -- (1) -- (2) -- (3) -- (4) -- (c) -- (6) -- (7) -- (8) -- (9) -- (10);
        \draw (c) -- (11a) -- (11) -- (12) -- (15);
        \draw (11) -- (13) -- (16);
        \draw (11) -- (14a) -- (14) -- (19);
        \draw (12) -- (17);
        \draw (13) -- (18);
        \draw (14) -- (20);
    \end{tikzpicture}}
    \end{subfigure}%
\begin{subfigure}{0.45 \textwidth}
    \centering
        \scalebox{0.85}{\begin{tikzpicture}[xscale=0.55, yscale=0.5]
         \node [fill=none] at (0.25, -0.4) {\footnotesize $c$};
         \node [fill=none] at (-1.1, 3.6) {\footnotesize \color{blue}$S$};
         \draw [blue] plot [smooth cycle, tension=0.6] coordinates {(0,0.4) (-1.7,2.4) (-1.3,4.1) (0, 4.7) (1.3,4.1) (1.7,2.4)};
    	\draw [decorate,decoration={brace,amplitude=5pt},xshift=0pt,yshift=0pt] (0,-0.6) -- (-5,-0.6) node [black,midway,xshift=0pt,yshift=-14pt]{\footnotesize $P_{\frac{k-1}{2}+1}$};
	\node [fill=none] at (2, 0) {$\in \mathcal{F'}$};
    
        \draw node[style=vertex](0) at (-5,0) {};
        \draw node[style=vertex](1) at (-4,0) {};
        \draw node[style=vertex](2) at (-3,0) {};
        \draw node[style=vertex](3) at (-2,0) {};
        \draw node[style=vertex](4) at (-1,0) {};
        \draw node[style=vertex](c) at (0,0) {};
        
        \draw node[style=vertex](11a) at (0, 0.7) {};
        \draw node[style=vertex](11) at (0,1.4) {};
        \draw node[style=vertex](12) at (-0.6,2) {};
        \draw node[style=vertex](13) at (0.6,2) {};
        \draw node[style=vertex](14a) at (0, 2.45) {};
        \draw node[style=vertex](14) at (0,3.5) {};
        \draw node[style=vertex](15) at (-1.3,2.6) {};
        \draw node[style=vertex](16) at (1.3,2.6) {};
        \draw node[style=vertex](17) at (-0.6,2.8) {};
        \draw node[style=vertex](18) at (0.6,2.8) {};
        \draw node[style=vertex](19) at (-0.5,4.1) {};
        \draw node[style=vertex](20) at (0.5,4.1) {};        
        
        \draw (0) -- (1) -- (2) -- (3) -- (4) -- (c);
        \draw (c) -- (11a) -- (11) -- (12) -- (15);
        \draw (11) -- (13) -- (16);
        \draw (11) -- (14a) -- (14) -- (19);
        \draw (12) -- (17);
        \draw (13) -- (18);
        \draw (14) -- (20);
    \end{tikzpicture}}
\end{subfigure}
    \caption{Elements of $\mathcal{F}$ and $\mathcal{F'}$.}
    \label{fig:families}
\end{figure}

\begin{figure}
\tikzstyle{vertex}=[circle, draw, fill=black, inner sep=0pt, minimum width=3.2pt]
    \centering
        \scalebox{0.85}{\begin{tikzpicture}[xscale=0.55, yscale=0.5]
        \node [fill=none] at (0.25, -0.4) {\footnotesize $c$};
        \node [fill=none] at (3.5, 3.6) {\footnotesize \color{blue}$B'$};
        \draw [decorate,decoration={brace,amplitude=5pt},xshift=0pt,yshift=0pt] (0,-0.6) -- (-5,-0.6) node [black,midway,xshift=0pt,yshift=-14pt]{\footnotesize $P_{\frac{k-1}{2}+1}$};
        \draw [blue] plot [smooth cycle, tension=0.7] coordinates {(0.4,0.45) (0.7,2.8) (2.5,4) (5.3, 4) (5.3, 2.4) (4,1.1) (2,0.6)};
    	 \draw[black] (0,0) -- (5,3.5);
    
        \draw node[style=vertex](0) at (-5,0) {};
        \draw node[style=vertex](1) at (-4,0) {};
        \draw node[style=vertex](2) at (-3,0) {};
        \draw node[style=vertex](3) at (-2,0) {};
        \draw node[style=vertex](4) at (-1,0) {};
        \draw node[style=vertex](c) at (0,0) {};
        \draw node[style=vertex](6) at (1,0) {};
        \draw node[style=vertex](7) at (2,0) {};
        \draw node[style=vertex](8) at (3,0) {};
        \draw node[style=vertex](9) at (4,0) {};
        \draw node[style=vertex](10) at (5,0) {};
        
        \draw node[style=vertex](11) at (2,1.4) {};
        \draw node[style=vertex](12) at (1.6,2) {};
        \draw node[style=vertex](13) at (2.2,2.4) {};
        \draw node[style=vertex](14) at (3,2.1) {};
        \draw node[style=vertex](15) at (0.9,1.9) {};
        \draw node[style=vertex](16) at (2.6,3.2) {};
        \draw node[style=vertex](17) at (1.2,2.7) {};
        \draw node[style=vertex](18) at (1.9,3.2) {};
        \draw node[style=vertex](19) at (4,2) {};
        \draw node[style=vertex](20) at (4,2.8) {};   
        \draw node[style=vertex](21) at (1,0.7) {}; 
        \draw node[style=vertex](22) at (5,3.5) {};
        
        \draw node[style=vertex](23) at (-0.85, 0.5) {};
        \draw node[style=vertex](24) at (-1.7, 1) {};
        \draw node[style=vertex](25) at (-2.55, 1.5) {};
        \draw node[style=vertex](26) at (-3.4, 2) {};
        \draw node[style=vertex](27) at (-4.25, 2.5) {};
        
        \draw[black] (0) -- (1) -- (2) -- (3) -- (4) -- (c);
        \draw[black] (c) -- (6) -- (7) -- (8) -- (9) -- (10);
        \draw[black] (c) -- (23) -- (24) -- (25) -- (26) -- (27); 
        \draw (11) -- (12) -- (15);
        \draw (11) -- (13) -- (16);
        \draw (14) -- (19);
        \draw (12) -- (17);
        \draw (13) -- (18);
    \end{tikzpicture}}
    \caption{A tree containing three longest paths that avoid $B'$ (so $\pi_{B}= 3$), and three longest paths that use $B'$ consisting of a $P_{(k-1)/2+1}$ outside $B'$ and a $P_{(k-1)/2}$ inside (so $\tau_{B} = 3$).}
    \label{fig:counting}
\end{figure}

We now consider the case where one of the branches at the centre of $T$ has at least $\ell - k$ vertices. This is so many, in fact, that we can find a card showing all the other branches at the centre in their entirety, which then reduces the problem to reconstructing the large branch. In order to do this, we will move the ``centre" one step inside the branch and continue doing this until no branch at the new centre is too big. This collection of branches can be reconstructed by essentially applying the proof of the previous lemma with minor modifications. Importantly, the condition that $T$ has small diameter ensures that we do not have to take too many steps away from the true centre.

The following lemma sets up for this process. We shall call a branch \emph{$i$-heavy} if it contains at least $\ell - k -i$ vertices (a \emph{heavy} branch is 0-heavy), and say it is \emph{outward} if it does not contain the centre of the tree. When we wish to talk about a branch at a vertex within a specific card, we will call it a \emph{partial branch} to emphasise that it need to not be a branch of $T$. Recall that $r:=n-\ell$.

\begin{lemma}\label{lemma:branchrecog}
Let $T$ be a tree with even diameter $k - 1$ (where $k < \ell - \frac{2n-1}{3}$ is odd) and central vertex $c$, and suppose we are given only the cards in $\mathcal{D}_\ell(T)$ that contain a copy of $P_k$. For any $0\leq i\leq (k-1)/2$, 
\begin{enumerate}[label=(\roman*)]
\itemsep=0mm
\item each vertex can have at most one $i$-heavy branch;
\item there is at most one vertex $c_i$ at distance $i$ from $c$ with an $i$-heavy outward branch;
\item we can recognise whether there is a vertex $c_i$ at distance $i$ from $c$ with an $i$-heavy outward branch;
\item if there is such a $c_i$, then we can find a card among those we are given on which we can identify $c_i$ and the root of its $i$-heavy branch, and all smaller branches at $c_i$ are present in their entirety. In particular, we can completely determine the isomorphism classes of all of these smaller branches. 
\end{enumerate}
\end{lemma}
\begin{proof}
Since $i\leq \frac{k-1}{2}$ and $k < \ell - \frac{2n-1}{3}$ by assumption, we first deduce that 
\[\ell - k -i \geq \ell-\frac{3k-1}{2} > \ell - \frac{3\big(\ell-\frac{(2n-1)}{3}\big)-1}{2}= \frac{2n-\ell}{2} > \frac{n}{2}.\]
This proves (i), as the branches at a vertex are pairwise disjoint. Similarly, if two distinct vertices $c_i$ and $c'_i$ are both at distance $i$ from $c$, then the only branch at $c_i$ that can share a vertex with a branch at $c'_i$ is that containing $c$. Thus, the previous calculation also proves (ii). 

For (iii), suppose that $T$ does have a vertex $c_i$ at distance $i$ from $c$ with an $i$-heavy branch $B$ not containing $c$. The subtree formed by taking a $P_k$ together with the path of length $i$ from $c$ to $c'$ and any $(\ell-k-i)$-vertex subtree of $B$ containing the root has at most $k+i+(\ell-k-i) = \ell$ vertices. On the other hand, if a card $C$ has a subtree with a $P_k$ (allowing us to identify $c$) and a vertex $c_i$ at distance $i$ from $c$ with a partial outward branch that has at least $\ell-k-i$ vertices, then we would be done. It follows that $T$ has such a vertex $c_i$ and $i$-heavy branch if and only if it has a card containing a subtree of this form. 

Assuming that there exist $c_i$ and $B$ as above, we claim that the desired card in (iv) can be found as follows: from among the cards we have (all with a copy of $P_k$ so we can identify $c$), take a connected card $C$ in which the maximum number of vertices in any partial outward branch at any vertex with distance $i$ from $c$ is as small as possible. 
There are only $r + k + i$ vertices not in $B$, so $C$ must still see at least $\ell- r - k - i$ vertices of $B$. On the other hand, every other partial branch at $c_i$ has at most $r+k+i$ vertices, which is less than $\ell-k-i -r$ since
\[
r+k+i \leq n -\ell + k+\frac{k-1}{2} < \frac{2n-2\ell+3(\ell - \frac{2n-1}{3})-1}{2} = \frac{\ell}{2}.
\] 
This means that we can identify the vertex $c_i$ as the unique (by (i) and (ii)) vertex at the correct distance from $c$ with a partial outward branch of size at least $\ell-k-i -r$, and the root of this partial branch is the root of the $i$-heavy branch in $T$. Moreover, by the minimality of the count used to select $C$, all other partial branches at $c_i$ must actually be present in their entirety; that is, they are isomorphic to the non-$i$-heavy branches at $c_i$ in $T$.
\end{proof}

\begin{lemma}\label{lemma:oneheavy}
If $T$ is a tree of diameter $k -1$ (where $k < \ell - \frac{2n-1}{3}$ is odd) and central vertex $c$, then $T$ is reconstructible from the subset of the $\ell$-deck consisting only of cards which contain a copy of $P_k$.
\end{lemma}
\begin{proof}
With $i=0$ in Lemma \ref{lemma:branchrecog}, we can recognise whether there is a branch at $c$ with at least $\ell-k$ vertices. Let us suppose there is, since we are otherwise done by Lemma \ref{lemma:allsmall} (or equivalently by setting $j=0$ and proceeding with the present proof). 

To reconstruct the heavy branch at $c:=c_0$, we construct a sequence of vertices $c_0,c_1,c_2,\dots$ to act as new ``centres" until the branches at some $c_j$ are all small enough for us to apply Lemma~\ref{lem:BH_counting}. 
For the first step, let $c_1$ be the root of the heavy branch, which is adjacent to $c_0$. Applying Lemma \ref{lemma:branchrecog} with $i=1$, we can recognise whether any neighbour of $c_0$ has a $1$-heavy outward branch. If not, then the branches at $c_1$ all have less than $\ell-k-1$ vertices and we terminate with $j = 1$. 
Else if there is such a 1-heavy outward branch, then it follows from statement (ii) of the lemma that it must be at $c_1$. In addition, statement (iv) allows us to determine all but the 1-heavy branch at $c_1$. 

Now set $c_2$ to be the vertex in the 1-heavy branch that is adjacent to $c_1$ and repeat the argument. In the $i$th step, we terminate if every branch at $c_i$ has weight less than $ \ell - k - i$, and otherwise we completely determine all but the $i$-heavy branch and proceed by setting $c_{i+1}$ to be the root of this branch. The case at hand is recognisable by Lemma \ref{lemma:branchrecog} and we can also reconstruct the smaller branches at each step provided $i \leq (k-1)/2$. To see that this condition is maintained, we note that our procedure builds a path in $T$ with one endvertex at $c$. Since each step increases the length of this path by 1 and the longest path in $T$ contains $k$ vertices, we can take at most $(k-1)/2$ steps before terminating. 

Suppose the process terminates at the $j$th step, where $j\leq (k-1)/2$.
The remainder of the argument closely follows the proof of Lemma \ref{lemma:allsmall}. Let $\mathcal{G}$ be the family of $n$-vertex trees with diameter $k - 1$, and $\mathcal{F}$ be the family of graphs that can be constructed as follows. Let $i \in \{0, \dots, j-1\}$, let $v_1, \dots, v_k$ be the vertices in a $P_k$ and let $u_1, \dots, u_{j-i}$ be the vertices in a (disjoint) $P_{j-i}$. A graph in $\mathcal{F}$ is formed by adding an edge from $u_1$ to $v_{\frac{k+1}{2} + i}$, and then attaching a rooted tree $S$ which is not an end-rooted path to the vertex $u_{j-i}$. The condition that the attached tree is not a path ensures that it is easy to identify $P_k$ and the added tree in any $\mathcal{F}$-graph. An example is given in Figure \ref{fig:change_center}.
\begin{figure}
    \centering
    \begin{subfigure}{0.5 \textwidth}
    \centering
    \scalebox{0.85}{\begin{tikzpicture}[xscale=0.7, yscale=0.6]
    \draw (-4,0) -- (4,0);
    \draw (2,0) -- (2,1.1);
    \draw (2,1.9) -- (2,2.3);
    \draw(2, 2.3) -- (2, 3);
    \draw [blue] plot [smooth cycle, tension=0.5] coordinates {(2,2.7) (1.5,3.3) (1.4,3.9) (1.8, 4.2) (2.2, 4.2) (2.6,3.9) (2.5,3.3)};
    \node [fill=none] at (2, 3.5) {\color{blue}$S$};
    \node [fill=none] at (2, 1.7) {$\vdots$};
	\draw [decorate,decoration={brace,amplitude=5pt},xshift=0pt,yshift=0pt] (4,-0.8) -- (-4,-0.8) node [black,midway,xshift=0pt,yshift=-14pt]{\footnotesize $P_k$};
    
    \draw node[style=vertex,label=below:$c$](c) at (0,0) {};
    \draw node[style=vertex,label=below:$c_1$](c1) at (1,0) {};
    \draw node[style=vertex,label=below:$c_2$](c2) at (2,0) {};
    \draw node[style=vertex,label=left:$c_3$](c3) at (2,0.7) {};
    \draw node[style=vertex,label=left:$c_j$](cj) at (2,2.3) {};
    \draw node[style=vertex] (root) at (2,3) {};
    \end{tikzpicture}}
    \end{subfigure}%
    \begin{subfigure}{0.5 \textwidth}
    \centering
    \scalebox{0.85}{\begin{tikzpicture}[xscale=0.7, yscale=0.6]
    \draw (-4,0) -- (4,0);
    \draw (0,0) -- (0,1.1);
    \draw (0,1.9) -- (0,2.3);
    \draw (0, 2.3) -- (0, 3);
    \draw [blue] plot [smooth cycle, tension=0.5] coordinates {(0,2.78) (-0.5,3.3) (-0.6,3.9) (-0.2, 4.2) (0.2, 4.2) (0.6,3.9) (0.5,3.3)};
    \node [fill=none] at (0, 3.5) {\color{blue}$S$};
    \node [fill=none] at (0, 1.7) {$\vdots$};
	\draw [decorate,decoration={brace,amplitude=5pt},xshift=0pt,yshift=0pt] (4,-0.8) -- (-4,-0.8) node [black,midway,xshift=0pt,yshift=-14pt]{\footnotesize $P_k$};
    
    \draw node[style=vertex,label=below:$c$](c) at (0,0) {};
    \draw node[style=vertex,label=left:$c_1$](c1) at (0,0.7) {};
    \draw node[style=vertex,label=left:$c_j$](cj) at (0,2.3) {};
    \draw node[style=vertex] (root) at (0,3) {};
    \end{tikzpicture}}
    \end{subfigure}
    \caption{Potential elements of $\mathcal{F}$ along with their `moving centres'.}
    \label{fig:change_center}
\end{figure}
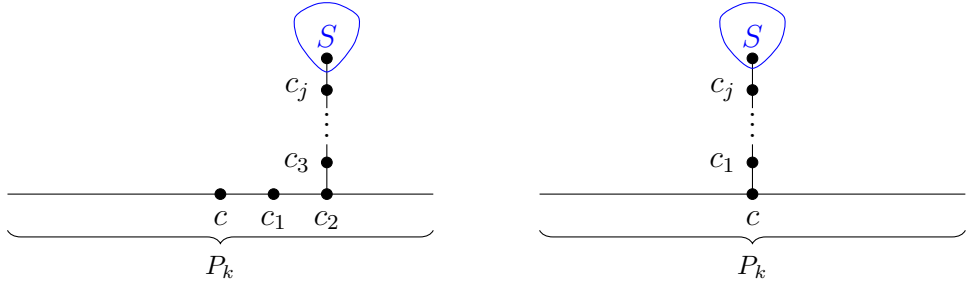

Each $\mathcal{F}$-subgraph of $G \in \mathcal{G}$ is contained in a unique maximal $\mathcal{F}$-subgraph, given by extending the tree attachment to the 
whole of the relevant branch at $u_{j-i}$. 
Applying Lemma~\ref{lem:BH_counting} allows us to determine the number of occurrences of each maximal $\mathcal{F}$-subgraph, as we did in the proof of Lemma~\ref{lemma:allsmall}.

At this point, each branch $B'$ has contributed one to the relevant count for each copy of $P_k$ which does not use $B'$, so we again need to determine the number of $P_k$ which use $B'$. 
This can be done using an identical argument to that in Lemma~\ref{lemma:allsmall} except replacing $c$ with $c_j$, replacing $P_{(k-1)/2+1} \frown S$ with $P_{(k-1)/2 + j  + 1} \frown S$ and suitably adjusting $S$.

We have now identified the total number of branches of each isomorphism class from vertices at distance $j$ from $c$ (except those which are end-rooted paths), although we do not know they are all branches at $c_j$. However, we have already reconstructed all of the tree except for the branches at $c_j$, so we can subtract the counts of all the appropriate branches not at $c_j$ from the total: the remainder must be attached at $c_j$. 

Finally, the end-rooted paths attached at $c_j$ can be reconstructed using the argument from the end proof of Lemma~\ref{lemma:allsmall}. 
\end{proof}

\lowdiam*

\begin{proof}
If $k$ is odd, then by Lemma~\ref{lemma:oneheavy} we can reconstruct any $n$-vertex tree with diameter $k-1$ using only the cards in its $\ell$-deck that contain a longest path provided $k < \ell - \frac{2n-1}{3}$, which is slightly better than the bound claimed in the theorem. The trees for which $k$ is even are then also reconstructible by Lemma~\ref{lem:parityreduction}, provided $k < \ell +1 - \frac{2(n+1)-1}{3} -1 = \ell-\frac{2n+1}{3}$.
\end{proof}

\section{Conclusion}\label{sec:end}
The example in Figure \ref{fig:counter_example} shows that the conjectured lower bound for reconstructing trees of $\floor{n/2} + 1$ is false for $n=13$, but the bound is still the best known for all other values of $n$. It may well be the case that the conjecture is asymptotically true, or even true exactly for large enough $n$.
\begin{problem}
Is there a function $\ell(n)=(1/2+o(1))n$ such that all $n$-vertex trees can be reconstructed from their $\ell(n)$-deck?
\end{problem}

For the problem of reconstructing the degree sequence, let $\ell=\ell(n)$ be the smallest integer such that the degree sequence of every $n$-vertex graph can be reconstructed from the $\ell$-deck. We have shown in Theorem \ref{thm:degree_small_cards} that $\ell(n)\le \sqrt{2n\log(2n)}+1$.  
It is easy to obtain a lower bound of form $\ell(n)=\Omega(\sqrt{\log n})$: indeed, each $\ell$-vertex graph appears at most $\binom{n}{\ell}$ times in the $\ell$-deck, so there are at most $(n^\ell)^{2^{\ell^2}}$ possible $\ell$-decks. There are  $\Omega(4^n/n)$ possible degree sequences as determined by Burns \cite{burns2007number}, and hence we need
$2^{\log_2(n)\ell 2^{\ell^2}}\geq 2^{2n-\log_2(n)}$,
which implies the bound.
By considering restricted graph classes, this can be slightly improved, but it would be interesting to see whether the {lower} bound can be improved to $n^\eps$ for some $\eps>0$.

In a different direction, it would be interesting to determine how large $\ell$ needs to be in order to recognise $k$-colourability of a graph on $n$ vertices from its $\ell$-deck. A special case of a result of Tutte \cite{horses} from 1979 states that the chromatic number of a graph is reconstructible when $\ell=n-1$, but nothing more is known in the direction of taking smaller cards. An interesting starting point would be to pinpoint the threshold for recognising whether a graph is bipartite (2-colourable).
In this case, a lower bound of $\lfloor n/2\rfloor$ follows from the example of Spinoza and West \cite{SW19} mentioned in the introduction (consider a path and the disjoint union of an odd cycle and a path). Manvel \cite{Manvel74} proved that the $(n-2)$-deck suffices, but it seems likely
that it should be possible to determine bipartiteness when a linear number of vertices are removed.  More generally, for fixed $k$, it may even be true that $k$-colourability is recognisable from the $cn$-deck for some $c=c(k)<1$.

\section*{Acknowledgements}
We would like to thank Doug West and the anonymous reviewer for their helpful suggestions.

\bibliography{reconstruction_edited}
\bibliographystyle{scott_edited}

\end{document}